\documentclass{article}%
\usepackage{amsmath}
\usepackage{amsfonts}
\usepackage{amssymb}
\usepackage{graphicx}
\usepackage[colorlinks=true, pdfstartview=FitV, linkcolor=blue,citecolor=red,
urlcolor=blue,a4paper]{hyperref}
\usepackage{color}%
\providecommand{\U}[1]{\protect\rule{.1in}{.1in}}
\newtheorem{theorem}{Theorem}

\newtheorem{corollary}[theorem]{Corollary}

\newtheorem{lemma}[theorem]{Lemma}

\newtheorem{proposition}[theorem]{Proposition}

\newenvironment{proof}[1][Proof]{\noindent\textbf{#1.} }{\ \rule{0.5em}{0.5em}}
\allowdisplaybreaks
\begin{document}

\title{Analytic properties of the parametric harmonic zeta function, with
applications to harmonic Stieltjes constants}
\author{Merve Kara \"{O}zt\"{u}rk\thanks{merve.kara.buket@gmail.com} and M\"{u}m\"{u}n
Can\thanks{mcan@akdeniz.edu.tr}\\{\small Department of Mathematics, Akdeniz University, 07058-Antalya,
T\"{u}rkiye}}
\date{}
\maketitle

\begin{abstract}
This paper investigates the analytic structure of the parametric harmonic zeta function $\zeta_{H}\left(  s,a,b\right)  $, a Dirichlet series associated with generalized harmonic numbers. We establish its meromorphic continuation to the complex plane and determine the residues at all of its poles. We then derive a Taylor expansion for  $\zeta_{H}\left(  s,a,b+t\right)$, leading to harmonic analogues and extensions of classical identities due to Landau, Singh–Verma, and Srivastava. We further develop a systematic theory of the associated harmonic Stieltjes constants by deriving explicit formulas, including previously unknown cases, together with a limit representation for the first-order constants. Finally, we construct the harmonic analogue of the classical digamma function, establish its principal analytic properties and its connection with the harmonic Stieltjes constants, and obtain Raabe-type formulas for the parametric harmonic zeta function, the harmonic Stieltjes constants, and the harmonic digamma function.


\textbf{Subclass}: {11M41, 30B50, 33B15, 30B40.}

\end{abstract}

\section{Introduction}

The harmonic zeta function, the Dirichlet series associated with the harmonic
numbers $H_{n}=1+1/2+\cdots+1/n$, is defined by
\[
\zeta_{H}\left(  s\right)  =\sum_{k=1}^{\infty}\frac{H_{k}}{k^{s}},\text{
}\operatorname{Re}\left(  s\right)  >1,
\]
and subject to many studies. Euler \cite[pp. 217--264]{E} has shown that
special values of the harmonic zeta function have relations to those of the
Riemann zeta values;%
\begin{equation}
2\zeta_{H}\left(  p\right)  =\left(  p+2\right)  \zeta\left(  p+1\right)
-\sum_{j=1}^{p-2}\zeta\left(  p-j\right)  \zeta\left(  j+1\right)  ,\text{
}p\in\mathbb{N}\backslash\left\{  1\right\}  , \label{ES}%
\end{equation}
where $\zeta\left(  s\right)  =\sum_{n=1}^{\infty}n^{-s}$ is the Riemann zeta
function. Apostol and Vu \cite{AV}, and Matsuoka \cite{Ma} have shown that the
function $\zeta_{H}$ can be holomorphically continued to the whole $s$-plane
except for the poles at $s=1$, $s=0$ and $s=1-2j,$ $j\in\mathbb{N}$. Later,
Boyadzhiev et al. \cite{BGP1}, and Candelpergher and Coppo \cite{CC} (see also
Can et al. \cite{CDK}) have dealt with the Laurent expansions at the poles of
$\zeta_{H}\left(  s\right)  $, and have recorded some formulas for the
corresponding coefficients. It should be mentioned that Alkan \cite{Al} has
demonstrated that certain real numbers and log-sine integrals can be closely
approximated using combinations of special values of the harmonic zeta
function and the Riemann zeta function. Recently, Alzer and Choi \cite{AC}
have introduced the following parametric Euler sums;%
\begin{align*}
S_{z,s}^{++}\left(  a,b\right)   &  =\sum_{n=1}^{\infty}\frac{\mathcal{H}%
_{n}^{\left(  z\right)  }\left(  a\right)  }{\left(  n+b\right)  ^{s}},\text{
}S_{z,s}^{+-}\left(  a,b\right)  =\sum_{n=1}^{\infty}\left(  -1\right)
^{n+1}\frac{\mathcal{H}_{n}^{\left(  z\right)  }\left(  a\right)  }{\left(
n+b\right)  ^{s}},\\
S_{z,s}^{-+}\left(  a,b\right)   &  =\sum_{n=1}^{\infty}\frac{\mathcal{A}%
_{n}^{\left(  z\right)  }\left(  a\right)  }{\left(  n+b\right)  ^{s}},\text{
}S_{z,s}^{--}\left(  a,b\right)  =\sum_{n=1}^{\infty}\left(  -1\right)
^{n+1}\frac{\mathcal{A}_{n}^{\left(  z\right)  }\left(  a\right)  }{\left(
n+b\right)  ^{s}},
\end{align*}
where%
\[
\mathcal{H}_{n}^{\left(  z\right)  }\left(  a\right)  =\sum_{k=1}^{n}\frac
{1}{\left(  k+a\right)  ^{z}}\text{ and }\mathcal{A}_{n}^{\left(  z\right)
}\left(  a\right)  =\sum_{k=1}^{n}\frac{\left(  -1\right)  ^{k-1}}{\left(
k+a\right)  ^{z}},
\]
$a,b\in\mathbb{C}\backslash\left\{  -1,-2,-3,\ldots\right\}  $ and $s,$
$z\in\mathbb{C}$ are adjusted so that the involved defining series converge.
They have investigated analytic continuations of $S_{z,s}^{++}\left(
a,a\right)  $, $S_{z,s}^{+-}\left(  a,a\right)  $, $S_{z,s}^{-+}\left(
a,a\right)  $ and $S_{z,s}^{--}\left(  a,a\right)  $ via summation formulas,
and have given shuffle relations. In particular, analytic continuations of
$S_{1,s}^{+-}\left(  0,0\right)  $, $S_{1,s}^{--}\left(  0,0\right)  $ and
$S_{1,s}^{-+}\left(  0,0\right)  $ have been investigated by Boyadzhiev et al.
in \cite{BGP2}.

Moreover, it is shown that the values $S_{1,m}^{-+}\left(  0,0\right)  ,$
$S_{1,2m}^{+-}\left(  0,0\right)  $ and $S_{1,2m}^{--}\left(  0,0\right)  $
(in \cite{Si} or see \cite{FlS}), $S_{1,m}^{++}\left(  0,1/2\right)  $ (in
\cite{SoCv}), $S_{1,2m}^{++}\left(  -1/2,0\right)  $ and $S_{1,2m}^{++}\left(
1/2,1/2\right)  $ (in \cite[Eqs. (8a) and (9a)]{J}) can be written in terms of
zeta values. In contrast to that for $m\in\mathbb{N}$, the values
$S_{1,2m+1}^{+-}\left(  0,0\right)  $, $S_{1,2m+1}^{--}\left(  0,0\right)  $,
$S_{1,2m+1}^{++}\left(  -1/2,0\right)  $ and $S_{1,2m+1}^{++}\left(
1/2,1/2\right)  $ do not directly admit any closed-form evaluations in terms
of zeta values as in (\ref{ES}). With this one, Alkan \cite{Al} shows that
$S_{1,m}^{++}\left(  -1/2,0\right)  $ can always be written as a combination
of log-sine integrals over $\left[  0,2\pi\right]  $ and $\left[
0,\pi\right]  $;
\[
I_{m,x}=\int\limits_{0}^{x}t^{m}\log\left(  2\sin\frac{t}{2}\right)  dt\text{
and }J_{m,x}=\int\limits_{0}^{x}t^{m}\log^{2}\left(  2\sin\frac{t}{2}\right)
dt.
\]
Quite recently, similar results have been recorded for the values
$S_{1,2m+1}^{+-}\left(  0,0\right)  $, $S_{1,2m+1}^{--}\left(  0,0\right)  $,
and $S_{1,2m+1}^{++}\left(  \frac{1}{2},\frac{1}{2}\right)  $ in \cite[p.
4]{CaKCD}. Furthermore, Karg\i n et al. \cite{KDCC} have investigated the
harmonic Stieltjes constants $\gamma_{H}\left(  m,a\right)  $, occurring in
the Laurent expansions of the function
\begin{equation}
\zeta_{H}\left(  s,a\right)  =S_{1,s}^{++}\left(  a-1,a-1\right)  =\sum
_{n=0}^{\infty}\frac{H_{n}\left(  a\right)  }{\left(  n+a\right)  ^{s}%
},\label{hhz}%
\end{equation}
where
\[
H_{n}\left(  a\right)  =\sum_{k=0}^{n}\frac{1}{k+a}.
\]
Among others, the authors have recorded closed-form expressions for the
harmonic Stieltjes constants $\gamma_{H,1}\left(  m,1,1\right)  $ and
$\gamma_{H,1}\left(  m,\frac{1}{2},\frac{1}{2}\right)  $ (see
(\ref{laurentzetahsab}) below). It should be mentioned that recent works
\cite{BGP1,BGP2}, \cite{CC} and \cite{KDCC} have addressed special cases:%
\[
\left\{
\begin{tabular}
[c]{ll}%
$\gamma_{H,1}\left(  m,a,a\right)  ,$ & $m\in\mathbb{N\cup}\left\{  0\right\}
,$\medskip\\
$\gamma_{H,-v}\left(  0,a,a\right)  $ and $\gamma_{H,-v}\left(  0,1,1\right)
,$ & $v\in\mathbb{N\cup}\left\{  0\right\}  .$%
\end{tabular}
\ \ \right.
\]
To the authors' knowledge, explicit representations for the constants
\[
\left\{
\begin{tabular}
[c]{ll}%
$\gamma_{H,-v}\left(  m,a,a\right)  $ & $m\geq1$ and $v\in\mathbb{N}%
\cup\left\{  0\right\}  ,$\medskip\\
$\gamma_{H,-v}\left(  m,1,1\right)  ,$ & $m\geq1$ and $v\in\mathbb{N}%
\cup\left\{  0\right\}  ,$%
\end{tabular}
\ \ \right.
\]
remain unknown. For the studies on $\zeta_{H}\left(  s\right)  =S_{1,s}%
^{++}\left(  0,0\right)  $ and on several types of Euler sums, see, for
example, \cite{BoBB,CKDS,EL,KCDC,LiQ,MCA,NKS,QSL,SoC}.

The primary objective of this study is to investigate the parametric harmonic
zeta function (or parametric Euler sum)%
\[
\zeta_{H}\left(  s,a,b\right)  :=S_{1,s}^{++}\left(  a-1,b-1\right)
=\sum_{n=0}^{\infty}\frac{H_{n}\left(  a\right)  }{\left(  n+b\right)  ^{s}},
\]
and related special functions, establishing new series expansions, summation
formulas, and explicit formulas.

We first establish its analytic continuation by using the Euler-Maclaurin
summation formula. More precisely, in Section \ref{secACon}, we prove the following:

\begin{theorem}
\label{teo1}Let $a,b\in\mathbb{R}\backslash\left\{  0,-1,-2,-3,\ldots\right\}
$. The function $\zeta_{H}\left(  s,a,b\right)  $ can be analytically
continued to the region $\mathbb{C}\backslash\left\{  k\in\mathbb{Z}%
:k\leq1\right\}  $. Moreover, $\zeta_{H}\left(  s,a,b\right)  $ has a double
pole at $s=1$ and simple poles at $s=-v,$ $v\in\mathbb{N}\cup\left\{
0\right\}  $ with residues $\gamma_{0}\left(  a\right)  $, and
\[
\left(  \frac{a-b}{v+1}+\frac{1}{2}\right)  \left(  b-a\right)  ^{v}%
-\sum\limits_{m=1}^{r}\frac{B_{2m}}{\left(  2m\right)  !}\left\langle
v\right\rangle _{2m-1}\left(  b-a\right)  ^{v-2m+1},
\]
respectively. Here $\gamma_{0}\left(  a\right)  $ is the generalized Stieltjes
constant defined by (\ref{hzLaurent}), $B_{n}$ is the $n$th Bernoulli number
defined by (\ref{Bn}), and $\left\langle x\right\rangle _{m}=x\left(
x-1\right)  \cdots\left(  x-\left(  m-1\right)  \right)  $ with the assumption
$\left\langle x\right\rangle _{0}=1$.
\end{theorem}

Secondly, we derive the Taylor series expansion for $\zeta_{H}\left(
s,a,b+t\right)  $ in a neighborhood of $t=0$, which serves as a generating
function and a foundational tool for subsequent investigations. We present it
in Section \ref{secHZeta}. Taking advantage of this Taylor expansion, we
generalize several classical formulas, such as Landau's formula
\begin{equation}
\zeta\left(  s\right)  =1+\frac{1}{s-1}-\sum_{j=1}^{\infty}\left(  -1\right)
^{j}\frac{\left(  s\right)  _{j}}{\left(  j+1\right)  !}\left\{  \zeta\left(
s+j\right)  -1\right\}  \label{Landau}%
\end{equation}
(cf. \cite[Ch. 3.2, Eq. (1)]{SC}) with $\left(  x\right)  _{m}=(-1)^{m}%
\left\langle -x\right\rangle _{m}$, Singh and Verma's formula
\begin{equation}
\zeta\left(  s\right)  =1+\frac{1}{2^{s+1}}\frac{s+3}{s-1}+\frac{1}{2}%
\sum_{j=1}^{\infty}\left(  -1\right)  ^{j+1}\frac{\left(  s\right)  _{j+1}%
}{\left(  j+2\right)  !}j\left(  \zeta\left(  s+j+1\right)  -1\right)
\label{singverma}%
\end{equation}
(\cite[Eq. (1.4)]{SV}), and Srivastava's formulas (11) and (15) in \cite[Ch.
3.2]{SC} (or see p. \pageref{formulas of Srivastava} below) to the context of
the harmonic zeta function.

In Section \ref{secStieltjes}, we deal with the harmonic Stieltjes constants
$\gamma_{H,-v}\left(  m,a,b\right)  $, $v\in\mathbb{N}\cup\left\{
-1,0\right\}  $, defined by\textbf{ }%
\begin{equation}
\zeta_{H}\left(  s,a,b\right)  =\frac{\omega_{1,-v}}{\left(  s+v\right)  ^{2}%
}+\frac{a_{-1,-v}\left(  a,b\right)  }{s+v}+\sum_{m=0}^{\infty}\left(
-1\right)  ^{m}\frac{\gamma_{H,-v}\left(  m,a,b\right)  }{m!}\left(
s+v\right)  ^{m},\label{laurentzetahsab}%
\end{equation}
where $\omega_{1,-v}=\left\{
\begin{array}
[c]{ll}%
1, & v=-1,\\
0, & v\not =-1,
\end{array}
\right.  $ and $a_{-1,-v}\left(  a,b\right)  $ stands for the residue of
$\zeta_{H}(s,\allowbreak a,\allowbreak b)$ at $s=-v$. Here, we fill the gap
identified above by deriving explicit expressions for $\gamma_{H,-v}\left(
m,a,b\right)  $, thereby encompassing all previously unknown cases, in terms
of derivatives of the Hurwitz and harmonic zeta functions. We also establish a
limit representation for the constants $\gamma_{H,1}\left(  m,a,b\right)  $.

In Section \ref{secDigamma}, motivated by the generating function identity for
the classical digamma function, we construct its parametric harmonic analogue,
the \emph{harmonic digamma function} $\psi_{H}\left(  a,b\right)  $, and
investigate its analytic properties. We establish difference equations,
derivative formulas, and a Taylor expansion showing that the values of the
parametric harmonic zeta function appear as its Taylor coefficients. In
particular, the harmonic digamma function is precisely the negative of the
first harmonic Stieltjes constan
\[
\psi_{H}\left(  a,b\right)  = -\gamma_{H,1}\left(  0,a,b\right)  ,
\]
paralleling the classical relationship between the digamma function and the
Stieltjes constants. The section concludes with a parametric harmonic analogue
of a classical summation formula, extending identities previously obtained for
the Riemann and Hurwitz zeta functions.

The final section of this paper is devoted to multiplication properties.
Specifically, we obtain multiplication formulas for the harmonic zeta function
and Stieltjes constants, which ultimately allow us to formulate a Raabe-type
relation for the harmonic digamma function.

\section{Proof of Theorem \ref{teo1}\label{secACon}}

For the proof, we utilize the following form of the Euler-Maclaurin summation
formula (see, e.g. \cite[p. 220]{SC}): Let $f$ be a function defined for
$t\geq x$ with continuous derivatives of order $2r+1$ and such that
$f^{\left(  k\right)  }\left(  t\right)  \rightarrow0$ as $t\rightarrow\infty
$, $\left(  k=0,1,\ldots,2r+1\right)  .$ Then%
\begin{align}
\sum\limits_{k=0}^{\infty}f\left(  x+k\right)   &  =\int\limits_{x}^{\infty
}f\left(  t\right)  dt+\frac{1}{2}f\left(  x\right)  -\sum\limits_{m=1}%
^{r}\frac{B_{2m}}{\left(  2m\right)  !}f^{\left(  2m-1\right)  }\left(
x\right) \nonumber\\
&  +\frac{1}{\left(  2r+1\right)  !}\int\limits_{0}^{\infty}\overline
{B}_{2r+1}\left(  t+x\right)  f^{\left(  2r+1\right)  }\left(  x+t\right)  dt.
\label{EM}%
\end{align}
Here $\overline{B}_{n}\left(  x\right)  =B_{n}\left(  x-\left\lfloor
x\right\rfloor \right)  $, $\left\lfloor x\right\rfloor $ is the greatest
integer not exceeding $x$, $B_{n}\left(  x\right)  $ is the $n$th Bernoulli
polynomial defined by%
\begin{equation}
\frac{te^{xt}}{e^{t}-1}=\sum_{n=0}^{\infty}\frac{B_{n}\left(  x\right)  }%
{n!}t^{n},\text{ }\left\vert t\right\vert <2\pi, \label{Bn}%
\end{equation}
with $B_{n}=B_{n}\left(  0\right)  $.

Let
\[
f\left(  t\right)  =\frac{1}{\left(  t+b\right)  ^{s}},\text{ }\left(
\operatorname{Re}(s)>1,b\in%
\mathbb{R}
^{+}\right)
\]
in (\ref{EM}). We then see that
\begin{align*}
\sum\limits_{k=0}^{\infty}\frac{1}{\left(  k+n+b\right)  ^{s}}  &
=\frac{\left(  n+b\right)  ^{1-s}}{s-1}+\frac{1}{2}\frac{1}{\left(
n+b\right)  ^{s}}+\sum\limits_{m=1}^{r}\frac{B_{2m}}{\left(  2m\right)
!}\frac{\left(  s\right)  _{2m-1}}{\left(  n+b\right)  ^{s+2m-1}}\\
&  -\frac{\left(  s\right)  _{2r+1}}{\left(  2r+1\right)  !}\int%
\limits_{0}^{\infty}\frac{\overline{B}_{2r+1}\left(  t+b\right)  }{\left(
n+b+t\right)  ^{s+2r+1}}dt.
\end{align*}
Hence,
\begin{align}
\zeta_{H}\left(  s,a,b\right)   &  =\frac{1}{s-1}\sum\limits_{n=0}^{\infty
}\frac{1}{\left(  n+a\right)  \left(  n+b\right)  ^{s-1}}+\frac{1}{2}%
\sum\limits_{n=0}^{\infty}\frac{1}{\left(  n+a\right)  \left(  n+b\right)
^{s}}\nonumber\\
&  +\sum\limits_{m=1}^{r}\frac{B_{2m}}{\left(  2m\right)  !}\left(  s\right)
_{2m-1}\sum\limits_{n=0}^{\infty}\frac{1}{\left(  n+a\right)  \left(
n+b\right)  ^{s+2m-1}}\nonumber\\
&  -\frac{\left(  s\right)  _{2r+1}}{\left(  2r+1\right)  !}\sum
\limits_{n=0}^{\infty}\frac{1}{n+a}\int\limits_{0}^{\infty}\frac{\overline
{B}_{2r+1}\left(  t+b\right)  }{\left(  n+b+t\right)  ^{s+2r+1}}dt. \label{**}%
\end{align}
Here, we use that
\begin{equation}
\sum\limits_{n=0}^{\infty}\frac{1}{\left(  n+a\right)  \left(  n+b\right)
^{s-1}}=\sum\limits_{k=0}^{\infty}\left(  b-a\right)  ^{k}\zeta\left(
s+k,b\right)  \label{*}%
\end{equation}
to obtain
\begin{align}
\zeta_{H}\left(  s,a,b\right)   &  =\frac{\zeta\left(  s,b\right)  }%
{s-1}+\frac{1}{s-1}\sum\limits_{k=1}^{\infty}\left(  b-a\right)  ^{k}%
\zeta\left(  s+k,b\right) \nonumber\\
&  +\frac{1}{2}\sum\limits_{j=0}^{\infty}\left(  b-a\right)  ^{j}\zeta\left(
s+j+1,b\right)  -\frac{\left(  s\right)  _{2r+1}}{\left(  2r+1\right)
!}\alpha\left(  s;a,b,r\right) \nonumber\\
&  +\sum\limits_{m=1}^{r}\frac{B_{2m}}{\left(  2m\right)  !}\left(  s\right)
_{2m-1}\sum\limits_{j=0}^{\infty}\left(  b-a\right)  ^{j}\zeta\left(
s+2m+j,b\right)  , \label{1a}%
\end{align}
where $\zeta\left(  s,b\right)  =\sum_{n=0}^{\infty}\left(  n+b\right)  ^{-s}$
is the Hurwitz zeta function and
\[
\alpha\left(  s;a,b,r\right)  =\sum\limits_{n=0}^{\infty}\int\limits_{0}%
^{\infty}\frac{\overline{B}_{2r+1}\left(  t+b\right)  }{\left(  n+a\right)
\left(  n+b+t\right)  ^{s+2r+1}}dt.
\]

The meromorphic continuation: The function $\overline{B}_{r}\left(  t\right)
$ is bounded on $%
\mathbb{R}
$, say $\left\vert \overline{B}_{2r+1}\left(  t\right)  \right\vert \leq M$.
Then
\[
\left\vert \int\limits_{0}^{\infty}\frac{\overline{B}_{2r+1}\left(
t+b\right)  }{\left(  n+b+t\right)  ^{s+2r+1}}dt\right\vert \leq\frac
{M}{\left(  \operatorname{Re}(s)+2r\right)  \left(  n+b\right)
^{\operatorname{Re}(s)+2r}},\text{\ }b>0,
\]
and hence
\[
\left\vert \alpha\left(  s;a,b,r\right)  \right\vert \leq\frac{M}%
{\operatorname{Re}(s)+2r}\sum\limits_{n=0}^{\infty}\frac{1}{\left(
n+a\right)  \left(  n+b\right)  ^{\operatorname{Re}(s)+2r}},\text{\ }a,b>0.
\]
Therefore $\alpha\left(  s;a,b,r\right)  $ is analytic in the half-plane
$\operatorname{Re}(s)>-2r$, and (\ref{1a}) provides the analytic continuation
of $\zeta_{H}\left(  s,a,b\right)  $ in the half-plane $\operatorname{Re}%
(s)>-2r$.

We observe from (\ref{1a}) that the singularities of $\zeta_{H}\left(
s,a,b\right)  $ arise from the poles of $\frac{\zeta\left(  s,b\right)  }%
{s-1}$, $\frac{\zeta\left(  s+k,b\right)  }{s-1}$, $\zeta\left(
s+j+1,b\right)  $ and $\zeta\left(  s+2m+j,b\right)  ,$ $k\geq1,$ $j\geq0,$
$1\leq m\leq r$. There is a second-order pole at $s=1$ arising from the term
$\zeta\left(  s,b\right)  /\left(  s-1\right)  ,$ and there are simple poles
at $s=-v,$ $v\in\mathbb{N}\cup\left\{  0\right\}  $. Consequently, the
function $\zeta_{H}\left(  s,a,b\right)  $ can be continued meromorphically to
the region $\mathbb{C}\backslash\left\{  v\in\mathbb{Z}:v\leq1\right\}  $.

To evaluate residues of $\zeta_{H}\left(  s,a,b\right)  $ at $s=-v\in
\mathbb{Z}$, $v\geq-1$, we make use of the Laurent series expansion of
$\zeta\left(  s,b\right)  $ in a neighborhood of $s=1:$%
\begin{equation}
\zeta\left(  s,b\right)  =\frac{1}{s-1}+\gamma_{0}\left(  b\right)
+\sum\limits_{n=1}^{\infty}\frac{\left(  -1\right)  ^{n}}{n!}\gamma_{n}\left(
b\right)  \left(  s-1\right)  ^{n}, \label{hzLaurent}%
\end{equation}
where the coefficients $\gamma_{n}\left(  b\right)  $ are called the
generalized Stieltjes constants.

By making the necessary calculations, we find that the residue of $\zeta
_{H}\left(  s,a,b\right)  $ at $s=1$ is%
\[
\gamma_{0}\left(  b\right)  +\sum\limits_{k=1}^{\infty}\left(  b-a\right)
^{k}\zeta\left(  k+1,b\right)  ,
\]
and the residues at $s=-v,$ $v\in\mathbb{N}\cup\left\{  0\right\}  $ are
\[
\left(  \frac{a-b}{v+1}+\frac{1}{2}\right)  \left(  b-a\right)  ^{v}%
-\sum\limits_{m=1}^{r}\frac{B_{2m}}{\left(  2m\right)  !}\left\langle
v\right\rangle _{2m-1}\left(  b-a\right)  ^{v-2m+1}.
\]
The equality%
\[
\gamma_{0}\left(  b\right)  +\sum\limits_{k=1}^{\infty}\left(  b-a\right)
^{k}\zeta\left(  k+1,b\right)  =\gamma_{0}\left(  a\right)
\]
follows from (\ref{*}). Thus, we complete the proof.

\section{Series involving $\zeta_{H}\left(  s,a,b\right)  $\label{secHZeta}}

The second result of the study is of the study is the Taylor series expansion
of $\zeta_{H}\left(  s,a,b+t\right)  $, which serves as a generating function
and a foundational tool for subsequent investigations.

\begin{theorem}
\label{hztaylor} For all values of $s\in\mathbb{C}\backslash\left\{
m\in\mathbb{Z}:m\leq1\right\}  $ we have
\begin{equation}
\zeta_{H}\left(  s,a,b+t\right)  =\sum\limits_{k=0}^{\infty}\frac{\left(
-1\right)  ^{k}\left(  s\right)  _{k}}{k!}\zeta_{H}\left(  s+k,a,b\right)
t^{k},\text{ }\left\vert t\right\vert \leq\left\vert b\right\vert ,\text{
}t\not =-b. \label{04}%
\end{equation}

\end{theorem}

\begin{proof}
For $\operatorname{Re}\left(  s\right)  >1$, (\ref{04}) follows from
(\ref{hhz}) and the binomial expansion. By analytic continuation, (\ref{04})
holds true for all values of $s\in\mathbb{C}\backslash\left\{  k\in
\mathbb{Z}:k\leq1\right\}  .$
\end{proof}

Following Srivastava \cite{Sr}, we order some consequences of (\ref{04}):

\begin{proposition}
Under the assumption of Theorem \ref{hztaylor}, we have
\begin{align}
\zeta_{H}\left(  s,a,b\right)   &  =\frac{1}{s-1}\sum\limits_{k=0}^{\infty
}\left(  b-a\right)  ^{k}\zeta\left(  s+k,b\right)  +\frac{1}{2}%
\sum\limits_{k=0}^{\infty}\left(  b-a\right)  ^{k}\zeta\left(  s+k+1,b\right)
\nonumber\\
&  +\sum\limits_{k=1}^{\infty}\frac{\left(  -1\right)  ^{k-1}\left(  s\right)
_{k+1}}{\left(  k+2\right)  !}\frac{k}{2}\zeta_{H}\left(  s+k+1,a,b\right)  .
\label{6a}%
\end{align}
provided that the series converges.
\end{proposition}

\begin{proof}
After some manipulations on the summation index in (\ref{04}), we
differentiate both sides of the resulting equation with respect to $t$ and
find that%
\begin{align}
&  \sum\limits_{k=1}^{\infty}\frac{\left(  -1\right)  ^{k-1}\left(  s\right)
_{k+1}}{\left(  k+2\right)  !}k\zeta_{H}\left(  s+k+1,a,b\right)
t^{k-1}\nonumber\\
&  =\left\{  \zeta_{H}\left(  s,a,b+t\right)  +\zeta_{H}\left(  s,a,b\right)
\right\}  t^{-2}\nonumber\\
&  +\frac{2}{s-1}\left\{  \zeta_{H}\left(  s-1,a,b+t\right)  -\zeta_{H}\left(
s-1,a,b\right)  \right\}  t^{-3},\text{ }0<\left\vert t\right\vert <\left\vert
b\right\vert , \label{4a}%
\end{align}
upon the use of $\frac{\partial}{\partial t}\left\{  \zeta_{H}\left(
s,a,b+t\right)  \right\}  =-s\zeta_{H}\left(  s+1,a,b+t\right)  $. Setting
$t=1$ in (\ref{4a}) and using that
\begin{equation}
\zeta_{H}\left(  s,a,b+1\right)  =\zeta_{H}\left(  s,a,b\right)
-\sum\limits_{k=0}^{\infty}\left(  b-a\right)  ^{k}\zeta\left(
s+k+1,b\right)  \label{*2}%
\end{equation}
yield (\ref{6a}).
\end{proof}

Similarly, we can deduce the following:

\begin{proposition}
We have%
\begin{equation}
\zeta_{H}\left(  s,a,b\right)  =\frac{1}{s-1}\sum\limits_{k=0}^{\infty}\left(
b-a\right)  ^{k}\zeta\left(  s+k,b\right)  -\sum\limits_{k=1}^{\infty}%
\frac{\left(  -1\right)  ^{k}\left(  s\right)  _{k}}{\left(  k+1\right)
!}\zeta_{H}\left(  s+k,a,b\right)  , \label{7a}%
\end{equation}
provided that the series converges.
\end{proposition}

In particular, for $a=b$, (\ref{6a}) reduces to (\ref{6b}) and (\ref{7a}) to
(\ref{10}) below;

\begin{corollary}
We have%
\begin{equation}
\zeta_{H}\left(  s,a\right)  =\frac{\zeta\left(  s,a\right)  }{s-1}+\frac
{1}{2}\zeta\left(  s+1,a\right)  -\sum\limits_{k=1}^{\infty}\left(  -1\right)
^{k}\frac{\left(  s\right)  _{k+1}}{\left(  k+2\right)  !}\frac{k}{2}\zeta
_{H}\left(  s+k+1,a\right)  \label{6b}%
\end{equation}
and%
\begin{equation}
\zeta_{H}\left(  s,a\right)  =\frac{\zeta\left(  s,a\right)  }{s-1}%
-\sum\limits_{k=1}^{\infty}\left(  -1\right)  ^{k}\frac{\left(  s\right)
_{k}}{\left(  k+1\right)  !}\zeta_{H}\left(  s+k,a\right)  , \label{10}%
\end{equation}
provided that the series converges.
\end{corollary}

It is worth noting that (\ref{6b}) and (\ref{10}) (and hence also (\ref{6a})
and (\ref{7a})) are generalizations of Srivastava's formulas%
\[
\zeta\left(  s,a\right)  =a^{-s}\left(  \frac{a}{s-1}+\frac{1}{2}\right)
+\sum\limits_{k=1}^{\infty}\left(  -1\right)  ^{k+1}\frac{\left(  s\right)
_{k+1}}{\left(  k+2\right)  !}\frac{k}{2}\zeta\left(  s+k+1,a\right)
\]
(cf. \cite[Ch. 3.2, Eq. (11)]{SC}) and\label{formulas of Srivastava}
\[
\zeta\left(  s,a\right)  =\frac{a^{1-s}}{s-1}-\sum\limits_{k=1}^{\infty
}\left(  -1\right)  ^{k}\frac{\left(  s\right)  _{k}}{\left(  k+1\right)
!}\zeta\left(  s+k,a\right)
\]
(cf. \cite[Ch. 3.2, Eq. (15)]{SC}), respectively.

The next result is the Landau's formula for the harmonic zeta function, which
follows from (\ref{10}) and (\ref{Landau}).

\begin{corollary}
We have
\begin{equation}
\zeta_{H}\left(  s\right)  =1+\frac{\zeta\left(  s\right)  }{s-1}-\sum
_{k=1}^{\infty}\left(  -1\right)  ^{k}\frac{\left(  s\right)  _{k}}{\left(
k+1\right)  !}\left\{  \zeta_{H}\left(  s+k\right)  -1\right\}  . \label{9a}%
\end{equation}

\end{corollary}

A further analogous formula to that of Singh and Verma (\ref{singverma}) is
given below.

\begin{corollary}
We have
\begin{align}
\zeta_{H}\left(  s\right)   &  =\frac{1}{2}+\frac{\zeta\left(  s\right)
-1+\left(  s+3\right)  2^{-s-1}}{s-1}+\frac{\zeta\left(  s+1\right)  }%
{2}\nonumber\\
&  +\sum\limits_{k=1}^{\infty}\left(  -1\right)  ^{k+1}\frac{\left(  s\right)
_{k+1}}{\left(  k+2\right)  !}\frac{k}{2}\left\{  \zeta_{H}\left(
s+k+1\right)  -1\right\}  . \label{9b}%
\end{align}

\end{corollary}

\section{Some results on the harmonic Stieltjes constants\label{secStieltjes}}

In this part, inspired by \cite{Coffey}, we investigate the constants arising
in the Laurent series expansion of the parametric harmonic zeta function at
its poles. We begin by summarizing the results derived from the use of
(\ref{7a}), explicit expressions for $\gamma_{H,-v}\left(  m,a,b\right)  $.

\begin{proposition}
The harmonic Stieltjes constants $\gamma_{H,1}\left(  m,a,b\right)  $ satisfy
\begin{align}
\gamma_{H,1}\left(  m,a,b\right)   &  =-\frac{\gamma_{m+1}\left(  b\right)
}{m+1}+\frac{\left(  -1\right)  ^{m}}{m+1}\sum\limits_{k=1}^{\infty}\left(
b-a\right)  ^{k}\zeta^{\left(  m+1\right)  }\left(  k+1,b\right) \nonumber\\
&  -\sum\limits_{k=1}^{\infty}\frac{\left(  -1\right)  ^{k+m}}{\left(
k+1\right)  !}\beta_{m,1}\left(  k,a,b\right)  , \label{gh1m}%
\end{align}
where%
\begin{align*}
\beta_{m,z}\left(  k,a,b\right)   &  =\left.  \frac{d^{m}}{ds^{m}}\left(
s\right)  _{k}\zeta_{H}\left(  s+k,a,b\right)  \right\vert _{s=z}.
\end{align*}

\end{proposition}

In particular, for $m=0$ and $m=1$, we have%
\begin{align*}
\gamma_{H,1}\left(  0,a,b\right)   &  =-\gamma_{1}\left(  a\right)
+\sum\limits_{k=1}^{\infty}\frac{\left(  -1\right)  ^{k}}{k}\zeta\left(
k+1,a\right)  \left(  b-a\right)  ^{k}+\sum\limits_{k=2}^{\infty}\frac{\left(
-1\right)  ^{k}}{k}\zeta_{H}\left(  k,a,b\right)  ,\\
\gamma_{H,1}\left(  1,a,b\right)   &  =-\frac{\gamma_{2}\left(  b\right)  }%
{2}-\frac{1}{2}\sum\limits_{k=2}^{\infty}\left(  b-a\right)  ^{k-1}%
\zeta^{\left(  2\right)  }\left(  k,b\right) \\
&  +\sum\limits_{k=2}^{\infty}\frac{\left(  -1\right)  ^{k-1}}{k}\left\{
H_{k-1}\zeta_{H}\left(  k,a,b\right)  +\zeta_{H}^{\prime}\left(  k,a,b\right)
\right\}  ,
\end{align*}
upon the use of
\[
-\gamma_{1}\left(  b\right)  +\sum\limits_{k=1}^{\infty}\left(  b-a\right)
^{k}\zeta^{\prime}\left(  k+1,b\right)  =-\gamma_{1}\left(  a\right)
+\sum\limits_{k=1}^{\infty}\frac{\left(  -1\right)  ^{k}}{k}\zeta\left(
k+1,a\right)  \left(  b-a\right)  ^{k}.
\]
In the case $a=b$, these results simplify. For example,%
\begin{align}
\sum\limits_{k=2}^{\infty}\frac{\left(  -1\right)  ^{k}}{k}\zeta_{H}\left(
k,a\right)   &  =\gamma_{H,1}\left(  0,a\right)  +\gamma_{1}\left(  a\right)
,\label{gh10}\\
\gamma_{H,1}\left(  m,a\right)  =-\frac{\gamma_{m+1}\left(  a\right)  }{m+1}
&  -\sum\limits_{k=1}^{\infty}\frac{\left(  -1\right)  ^{k+m}}{\left(
k+1\right)  !}\beta_{m,1}\left(  k,a,a\right)  , \label{gh11}%
\end{align}
where $\gamma_{H,1}\left(  m,a\right)  =\gamma_{H,1}\left(  m,a,a\right)  $
and $H_{k}^{\left(  p\right)  }=\sum\limits_{j=1}^{k}1/j^{p}$. Notice that
(\ref{gh10}) is an analog of the familiar result
\[
\sum\limits_{k=2}^{\infty}\frac{\left(  -1\right)  ^{k}}{k}\zeta\left(
k\right)  =\gamma_{0}\left(  1\right)  =\gamma.
\]
Moreover, we infer from (\ref{gh1m}) with $a=b=1$ and \cite[Eq. (21)]{KDCC}
that the series
\[
\sum\limits_{k=1}^{\infty}\frac{\left(  -1\right)  ^{k+m}}{\left(  k+1\right)
!}\beta_{m,1}\left(  k,1,1\right)
\]
can be written in terms of Riemann zeta values, Stieltjes constants
$\gamma_{m}\left(  1\right)  $, and the integrals%
\[
\int\limits_{0}^{\infty}\frac{e^{-x}\log\frac{1-e^{-x}}{x}}{1-e^{-x}}\left(
\log x\right)  ^{j}dx,\text{ }0\leq j\leq m.
\]

\begin{proposition}
The harmonic Stieltjes constants $\gamma_{H,0}\left(  m,a,b\right)  $ satisfy%
\begin{align*}
\gamma_{H,0}\left(  m,a,b\right)   &  =\left(  -1\right)  ^{m+1}m!\sum
_{n=0}^{m}\frac{\zeta^{\left(  n\right)  }\left(  0,b\right)  +\left(
-1\right)  ^{n}\gamma_{n}\left(  b\right)  \left(  b-a\right)  }{n!}\\
&  +\left(  -1\right)  ^{m+1}m!\left\{  b-a+\sum_{n=0}^{m}\frac{1}{n!}%
\sum\limits_{k=2}^{\infty}\zeta^{\left(  n\right)  }\left(  k,b\right)
\left(  b-a\right)  ^{k}\right\} \\
&  +\frac{1}{2}\rho\left(  m,a,b\right)  -\sum\limits_{k=2}^{\infty}%
\frac{\left(  -1\right)  ^{m+k}}{\left(  k+1\right)  !}\beta_{m,0}\left(
k,a,b\right)  ,
\end{align*}
where $\rho\left(  m,a,b\right)  =\left\{
\begin{array}
[c]{ll}%
\gamma_{0}\left(  a\right)  , & m=0,\\
-m\gamma_{H,1}\left(  m-1,a,b\right)  , & m\geq1.
\end{array}
\right.  $
\end{proposition}

The first two constants are
\begin{align*}
\gamma_{H,0}\left(  0,a,b\right)   &  =\frac{1}{2}\gamma_{0}\left(  a\right)
-\zeta\left(  0,b\right)  -\left(  \gamma_{0}\left(  a\right)  +1\right)
\left(  b-a\right)  ,\\
\gamma_{H,0}\left(  1,a,b\right)   &  =\zeta\left(  0,b\right)  +\zeta
^{\prime}\left(  0,b\right)  -\frac{1}{2}\gamma_{H,1}\left(  0,a,b\right)
+\left(  \gamma_{0}\left(  a\right)  -\gamma_{1}\left(  a\right)  +1\right)
\left(  b-a\right) \\
&  +\sum\limits_{k=1}^{\infty}\frac{\left(  -1\right)  ^{k}}{k}\zeta\left(
k+1,a\right)  \left(  b-a\right)  ^{k+1}+\sum\limits_{k=2}^{\infty}%
\frac{\left(  -1\right)  ^{k}}{\left(  k+1\right)  k}\zeta_{H}\left(
k,a,b\right)  .
\end{align*}
These formulas are further simplified for $a=b$;%
\[
\gamma_{H,0}\left(  0,a\right)  =\frac{1}{2}\gamma_{0}\left(  a\right)
-\zeta\left(  0,a\right)  ,
\]
which coincides with the result in \cite[Corollary 9]{KDCC} and in
\cite[Corollary 2]{BGP1} for $a=1$,
\begin{align}
\gamma_{H,0}\left(  1,a\right)   &  =\zeta\left(  0,a\right)  +\zeta^{\prime
}\left(  0,a\right)  -\frac{1}{2}\gamma_{H,1}\left(  0,a\right)
+\sum\limits_{k=2}^{\infty}\frac{\left(  -1\right)  ^{k}}{\left(  k+1\right)
k}\zeta_{H}\left(  k,a\right)  ,\label{gh01}\\
\gamma_{H,0}\left(  m,a\right)   &  =\left(  -1\right)  ^{m+1}m!\sum_{n=0}%
^{m}\frac{\zeta^{\left(  n\right)  }\left(  0,a\right)  }{n!}-\frac{m}%
{2!}\gamma_{H,1}\left(  m-1,a\right) \nonumber\\
&  +\sum\limits_{k=2}^{\infty}\frac{\left(  -1\right)  ^{m+1+k}}{\left(
k+1\right)  !}\beta_{m,0}\left(  k,a,a\right)  . \label{gh02}%
\end{align}

Employing (\ref{gh10}), (\ref{gh11}), (\ref{gh01}), and (\ref{gh02}) with the
use of
\[
\gamma_{H,1}\left(  0\right)  =\gamma_{H,1}\left(  0,1\right)  =\frac
{\gamma_{0}^{2}\left(  1\right)  +\zeta\left(  2\right)  }{2}%
\]
(see \cite[Eq. (6)]{CC} or \cite[Remark 2]{CDK}), we can derive closed-form
evaluations for certain alternating series:

\begin{corollary}
\label{cor1}The following series evaluations hold:%
\begin{align*}
\sum\limits_{k=2}^{\infty}\frac{\left(  -1\right)  ^{k}}{k}\zeta_{H}\left(
k\right)   &  =\gamma_{1}\left(  1\right)  +\frac{\gamma_{0}^{2}\left(
1\right)  +\zeta\left(  2\right)  }{2},\\
\sum\limits_{k=2}^{\infty}\frac{\left(  -1\right)  ^{k}}{k+1}\zeta_{H}\left(
k,a\right)  =\zeta\left(  0,a\right)  +  &  \zeta^{\prime}\left(  0,a\right)
+\gamma_{1}\left(  a\right)  +\frac{1}{2}\gamma_{H,1}\left(  0,a\right)
-\gamma_{H,0}\left(  1,a\right)  ,\\
2\sum\limits_{k=2}^{\infty}\frac{\left(  -1\right)  ^{k}}{k+1}\left\{
\zeta_{H}^{\prime}\left(  k,a\right)  +H_{k-1}\right.   &  \left.  \zeta
_{H}\left(  k,a\right)  \right\} \\
=\gamma_{H,0}\left(  2,a\right)  -\gamma_{H,1}\left(  1,a\right)  -  &
\gamma_{2}\left(  a\right)  +2\zeta\left(  0,a\right)  +2\zeta^{\prime}\left(
0,a\right)  +\zeta^{\prime\prime}\left(  0,a\right)  ,
\end{align*}
and%
\begin{align*}
&  3\sum\limits_{k=2}^{\infty}\frac{\left(  -1\right)  ^{k}}{k+1}\left\{
\zeta_{H}^{\prime\prime}\left(  k,a\right)  +2H_{k-1}\zeta_{H}^{\prime}\left(
k,a\right)  +\left(  \left(  H_{k-1}\right)  ^{2}-H_{k-1}^{\left(  2\right)
}\right)  \zeta_{H}\left(  k,a\right)  \right\} \\
&  \quad=\frac{3}{2}\gamma_{H,1}\left(  2,a\right)  -\gamma_{H,0}\left(
3,a\right)  +\gamma_{3}\left(  a\right)  +6\sum_{n=0}^{3}\frac{\zeta^{\left(
n\right)  }\left(  0,a\right)  }{n!}.
\end{align*}

\end{corollary}

It should be noted that Coppo \cite[Eq. (14)]{Co} also recorded the first
evaluation of Corollary \ref{cor1}.

\begin{proposition}
The harmonic Stieltjes constants $\gamma_{H,-v}\left(  m,a,b\right)  $,
$v\in\mathbb{N}$, satisfy
\begin{align}
&  \gamma_{H,-v}\left(  m,a,b\right) \nonumber\\
&  =\left(  -1\right)  ^{m+1}m!\sum_{n=0}^{m}\frac{\zeta^{\left(  n\right)
}\left(  -v,b\right)  +\left(  -1\right)  ^{n}\gamma_{n}\left(  b\right)
\left(  b-a\right)  ^{v+1}}{n!\left(  1+v\right)  ^{m-n+1}}\nonumber\\
&  -\sum\limits_{k=m+1}^{v}\frac{\left(  -1\right)  ^{m+k}}{\left(
k+1\right)  !}\frac{P_{k}^{\left(  m+1\right)  }\left(  -v\right)  }{\left(
m+1\right)  }a_{-1,k-v}\left(  a,b\right)  -\sum\limits_{k=v+2}^{\infty}%
\frac{\left(  -1\right)  ^{m+k}}{\left(  k+1\right)  !}\beta_{m,-v}\left(
k,a,b\right) \nonumber\\
&  +\chi\left(  m,v-1\right)  \frac{\left(  -1\right)  ^{v+m}}{\left(
v+2\right)  !}\frac{P_{v}^{\left(  m+1\right)  }\left(  -v\right)  }{m+1}%
+\chi\left(  m,v\right)  \frac{\left(  -1\right)  ^{v+m}}{\left(  v+2\right)
!}\gamma_{0}\left(  a\right)  P_{v}^{\left(  m\right)  }\left(  -v\right)
\nonumber\\
&  -\left(  -1\right)  ^{m}m!\left\{  \frac{\left(  b-a\right)  ^{v+1}%
}{\left(  1+v\right)  ^{m+2}}-\sum_{n=0}^{m}\sum
\limits_{\substack{k=0\\k\not =v}}^{\infty}\frac{\zeta^{\left(  n\right)
}\left(  k-v+1,b\right)  }{n!\left(  1+v\right)  ^{m-n+1}}\left(  b-a\right)
^{k+1}\right\} \nonumber\\
&  +\sum\limits_{k=1}^{v}\frac{\left(  -1\right)  ^{k+1}}{\left(  k+1\right)
!}\sum\limits_{\mu=0}^{\min\left\{  k,m\right\}  }\binom{m}{\mu}\left(
-1\right)  ^{\mu}P_{k}^{\left(  \mu\right)  }\left(  -v\right)  \gamma
_{H,k-v}\left(  m-\mu,a,b\right) \nonumber\\
&  +m\frac{\left(  -1\right)  ^{v-1}}{\left(  v+2\right)  !}\sum
\limits_{\mu=0}^{\min\left\{  v,m-1\right\}  }\binom{m-1}{\mu}\left(
-1\right)  ^{\mu}P_{v}^{\left(  \mu\right)  }\left(  -v\right)  \gamma
_{H,1}\left(  m-1-\mu,a,b\right)  , \label{39}%
\end{align}
where $P_{n}\left(  s\right)  =\left(  s\right)  _{n}=\sum_{\mu=0}^{n}%
\frac{P_{n}^{\left(  \mu\right)  }\left(  -v\right)  }{\mu!}\left(
s+v\right)  ^{\mu}$ and $\chi\left(  m,v\right)  =\left\{
\begin{array}
[c]{ll}%
1, & 0\leq m\leq v,\\
0, & \text{otherwise.}%
\end{array}
\right.  $
\end{proposition}

In the case $a=b$, (\ref{39}) simplifies. For example, for $v=1$ and $m=0,1,$
\begin{equation}
\gamma_{H,-1}\left(  0,a\right)  =\frac{1}{6}\gamma_{0}\left(  a\right)
-\frac{1}{2}\zeta\left(  -1,a\right)  -\frac{1}{2}\gamma_{H,0}\left(
0,a\right)  +\frac{1}{12} \label{41}%
\end{equation}
and
\begin{align}
\gamma_{H,-1}\left(  1,a\right)   &  =\frac{1}{6}\gamma_{0}\left(  a\right)
-\frac{1}{6}\gamma_{H,1}\left(  0,a\right)  -\frac{1}{2}\gamma_{H,0}\left(
0,a\right)  -\frac{1}{2}\gamma_{H,0}\left(  1,a\right) \nonumber\\
&  +\frac{\zeta\left(  -1,a\right)  }{4}+\frac{\zeta^{\prime}\left(
-1,a\right)  }{2}+\sum\limits_{k=2}^{\infty}\frac{\left(  -1\right)  ^{k}%
}{k\left(  k+1\right)  \left(  k+2\right)  }\zeta_{H}\left(  k,a\right)
.\nonumber
\end{align}
Remark that (\ref{41}) coincides with the result in \cite[Corollary 10]{KDCC}
upon correcting a misprint; the term $\frac{\gamma_{0}\left(  1\right)  }{12}$
in \cite[Corollary 3]{BGP1} must be read as $-\frac{\gamma_{0}\left(
1\right)  }{12}$ which implies that the corresponding term $\frac{\gamma
_{0}\left(  1\right)  }{6}$ in \cite[Corollary 10]{KDCC} is null. As mentioned
in the introdulctory section, for $m\geq1$ and $v\in\mathbb{N}\cup\left\{
0\right\}  $, no explicit formulas for $\gamma_{H,-v}\left(  m,a\right)  $ and
$\gamma_{H,-v}\left(  m,1\right)  $ have been previously established in the
literature.\bigskip

In a similar way, we can deduce explicit formulas for $\gamma_{H,-v}\left(
m,a,b+t\right)  .$ In this case, we employ (\ref{04}) and find

\noindent$\bullet$ for $\gamma_{H,1}\left(  m,a,b+t\right)  $;%
\[
\gamma_{H,1}\left(  m,a,b+t\right)  =\gamma_{H,1}\left(  m,a,b\right)
+\sum_{n=1}^{\infty}\left(  -1\right)  ^{n}\beta_{m,1}\left(  n,a,b\right)
\frac{t^{n}}{n!},
\]
\noindent$\bullet$ for $\gamma_{H,0}\left(  m,a,b+t\right)  $;%
\begin{align*}
\gamma_{H,0}\left(  0,a,b+t\right)   &  =\gamma_{H,0}\left(  0,a,b\right)
-\gamma_{0}\left(  a\right)  t,\\
\gamma_{H,0}\left(  m,a,b+t\right)   &  =\gamma_{H,0}\left(  m,a,b\right)
+mt\gamma_{H,1}\left(  m-1,a,b\right) \\
&  \ +\sum_{n=2}^{\infty}\left(  -1\right)  ^{n}\beta_{m,0}\left(
n,a,b\right)  \frac{t^{n}}{n!},\text{ }m\geq1,
\end{align*}
\noindent$\bullet$ for $\gamma_{H,-v}\left(  m,a,b+t\right)  $, $v\in
\mathbb{N}$;%
\begin{align*}
\gamma_{H,-v}\left(  0,a,b+t\right)   &  =\gamma_{H,-v}\left(  0,a,b\right)
+\left(  P_{v}^{\left(  1\right)  }\left(  -v\right)  +\gamma_{0}\left(
a\right)  P_{v}\left(  -v\right)  \right)  \frac{\left(  -t\right)  ^{v+1}%
}{\left(  v+1\right)  !}\\
&  +\sum_{n=1}^{v}\left(  a_{-1,n-v}\left(  a,b\right)  P_{n}^{\left(
1\right)  }\left(  -v\right)  +\gamma_{H,n-v}\left(  0,a,b\right)
P_{n}\left(  -v\right)  \right)  \frac{\left(  -t\right)  ^{n}}{n!}.
\end{align*}
and for $m\geq1$,%
\begin{align*}
&  \gamma_{H,-v}\left(  m,a,b+t\right) \\
&  =\gamma_{H,-v}\left(  m,a,b\right)  +\frac{\left(  -1\right)  ^{m}}%
{m+1}\sum_{n=m+1}^{v}\left(  -1\right)  ^{n}a_{-1,n-v}\left(  a,b\right)
P_{n}^{\left(  m+1\right)  }\left(  -v\right)  \frac{t^{n}}{n!}\\
&  +\left(  -1\right)  ^{m}\left(  \chi\left(  m,v\right)  \gamma_{0}\left(
a\right)  P_{v}^{\left(  m\right)  }\left(  -v\right)  +\chi\left(
m,v-1\right)  \frac{P_{v}^{\left(  m+1\right)  }\left(  -v\right)  }%
{m+1}\right)  \frac{\left(  -t\right)  ^{v+1}}{\left(  v+1\right)  !}\\
&  +\sum_{n=1}^{v}\frac{\left(  -t\right)  ^{n}}{n!}\sum\limits_{\mu=0}%
^{\min\left\{  n,m\right\}  }\binom{m}{\mu}\left(  -1\right)  ^{\mu}%
P_{n}^{\left(  \mu\right)  }\left(  -v\right)  \gamma_{H,n-v}\left(
m-\mu,a,b\right) \\
&  +m\frac{\left(  -t\right)  ^{v+1}}{\left(  v+1\right)  !}\sum
\limits_{\mu=0}^{\min\left\{  v,m-1\right\}  }\binom{m-1}{\mu}\left(
-1\right)  ^{\mu-1}P_{v}^{\left(  \mu\right)  }\left(  -v\right)  \gamma
_{H,1}\left(  m-1-\mu,a,b\right) \\
&  +\sum_{n=v+2}^{\infty}\left(  -1\right)  ^{n}\beta_{m,-v}\left(
n,a,b\right)  \frac{t^{n}}{n!}.
\end{align*}

For $t=1$, these expressions provide a closed-form evaluation formula for
$\gamma_{H,-v}\left(  m,a,b+1\right)  -\gamma_{H,-v}\left(  m,a,b\right)  ,$
i.e. a difference formula for $\gamma_{H,-v}\left(  m,a,b\right)  $. Moreover,
differentiating with respect to $t$ gives, for the cases $-v=1$ and $v=0$,
\[
\left.  \frac{\partial^{j}}{\partial t^{j}}\gamma_{H,1}\left(  m,a,b+t\right)
\right\vert _{t=0}=\left(  -1\right)  ^{j}\beta_{m,1}\left(  j,a,b\right)  ,
\]
and%
\begin{align*}
\left.  \frac{\partial^{j}}{\partial t^{j}}\gamma_{H,0}\left(  0,a,b+t\right)
\right\vert _{t=0}  &  =\left\{
\begin{array}
[c]{ll}%
-\gamma_{0}\left(  a\right)  , & j=1,\text{ }\smallskip\\
0, & j>1,
\end{array}
\right. \\
\left.  \frac{\partial^{j}}{\partial t^{j}}\gamma_{H,0}\left(  m,a,b+t\right)
\right\vert _{t=0}  &  =\left\{
\begin{array}
[c]{ll}%
m\gamma_{H,1}\left(  m-1,a,b\right)  , & j=1,\text{ }m\geq1,\smallskip\\
\left(  -1\right)  ^{j}\beta_{m,0}\left(  j,a,b\right)  , & j>1,\text{ }%
m\geq1.
\end{array}
\right.
\end{align*}

\subsection{A limit representation for $\gamma_{H,1}\left(  m,a,b\right)  $}

To prove the limit representation for $\gamma_{H,1}\left(  m,a,b\right)  $, we
need some auxiliary results. The proofs of these results follow by modifying
the method of \cite{BrB}.

\begin{lemma}
\label{4}For $\operatorname{Re}\left(  s\right)  >1$ we have%
\[
\zeta_{H}\left(  s,a,b\right)  =s\int_{1-b}^{\infty}\left(  \sum\limits_{n\leq
t}H_{n}\left(  a\right)  \right)  \left(  t+b\right)  ^{-s-1}dt.
\]

\end{lemma}

Let%
\begin{align}
E\left(  x,a\right)   &  =\sum\limits_{n\leq x}H_{n}\left(  a\right)  -\left(
x+a\right)  \log\left(  x+a\right)  -\left(  x+a\right)  \left(  \gamma\left(
0,a\right)  -1\right) \nonumber\\
&  =O\left(  \log\left(  x+a\right)  \right)  . \label{E(x,a)}%
\end{align}

\begin{lemma}
\label{5}For $\operatorname{Re}\left(  s\right)  >0$ we have%
\begin{align}
f\left(  s,a,b\right)   &  =s\int_{1-b}^{\infty}\left(  t+b\right)
^{-s-1}E\left(  t,a\right)  dt\label{f(s,a,b)}\\
&  =\left(  b-a-1\right)  \left(  \gamma\left(  0,a\right)  -1+\log\left(
1+a-b\right)  \right) \nonumber\\
&  +\sum\limits_{m=0}^{\infty}\left\{  \frac{\left(  -1\right)  ^{m}%
\gamma_{H,1}\left(  m,a,b\right)  }{m!}+\sum\limits_{n=0}^{\infty}%
\frac{\left(  -1\right)  ^{m}\left(  b-a\right)  ^{n+1}}{\left(  n+1\right)
^{m+2}}\right\}  \left(  s-1\right)  ^{m}.\nonumber
\end{align}

\end{lemma}

\begin{theorem}
\label{teo 2}Let $m$ be a non-negative integer and let $u<0$. Then, we have%
\begin{align*}
\sum\limits_{n\leq x}\frac{H_{n}\left(  a\right)  \log^{m}\left(  n+b\right)
}{\left(  n+b\right)  ^{-u}}  &  =\int_{1-b}^{x}\frac{\log\left(  t+a\right)
\log^{m}\left(  t+b\right)  +\gamma\left(  0,a\right)  \log^{m}\left(
t+b\right)  }{\left(  t+b\right)  ^{-u}}dt\\
&  +\left(  -1\right)  ^{m}f^{\left(  m\right)  }\left(  -u,a,b\right)
+\delta_{m}+o\left(  1\right)  ,
\end{align*}
where $f\left(  s,a,b\right)  $ is the function introduced in Lemma \ref{5},
and
\[
\delta_{m}=\left\{
\begin{array}
[c]{cc}%
\left(  1-b+a\right)  \left(  \gamma\left(  0,a\right)  -1+\log\left(
1-b+a\right)  \right)  , & m=0,\\
0, & \text{otherwise.}%
\end{array}
\right.
\]

\end{theorem}

We now state and prove the aforementioned limit representation.

\begin{theorem}
\label{teo 1}Let $0<b-a\leq1$. Then, the coefficients $\gamma_{H,1}\left(
m,a,b\right)  $ have the limit representation%
\[
\gamma_{H,1}\left(  m,a,b\right)  =\lim_{x\longrightarrow\infty}\left(
\sum\limits_{n\leq x}\frac{H_{n}\left(  a\right)  \log^{m}\left(  n+b\right)
}{n+b}-\frac{\log^{m+2}\left(  x+b\right)  }{m+2}-\gamma\left(  0,a\right)
\frac{\log^{m+1}\left(  x+b\right)  }{m+1}\right)  .
\]

\end{theorem}

\begin{proof}
We sketch the proof. It follows from Lemma (\ref{5}) that
\begin{equation}
f\left(  1,a,b\right)  =\left(  b-a-1\right)  \left(  \gamma\left(
0,a\right)  -1+\log\left(  1+a-b\right)  \right)  +\gamma_{H,1}\left(
0,a,b\right)  +\sum\limits_{n=1}^{\infty}\frac{\left(  b-a\right)  ^{n}}%
{n^{2}} \label{m=0}%
\end{equation}
and
\[
\left(  -1\right)  ^{m}f^{\left(  m\right)  }\left(  1,a,b\right)
=\gamma_{H,1}\left(  m,a,b\right)  +m!\sum\limits_{n=1}^{\infty}\frac{\left(
b-a\right)  ^{n}}{n^{m+2}}\text{ for }m>0.
\]
where
\[
a_{m}=\frac{\left(  -1\right)  ^{m}}{m!}\gamma_{H,1}\left(  m,a,b\right)
\]
Setting $u=-1$ in Theorem \ref{teo 2} we find that
\begin{align*}
\gamma_{H,1}\left(  m,a,b\right)   &  =\sum\limits_{n\leq x}\frac{H_{n}\left(
a\right)  \log^{m}\left(  n+b\right)  }{n+b}-m!\sum\limits_{n=1}^{\infty}%
\frac{\left(  b-a\right)  ^{n}}{n^{m+2}}-\delta_{m}+o\left(  1\right) \\
&  \ -\int_{1-b}^{x}\frac{\log\left(  t+a\right)  \log^{m}\left(  t+b\right)
}{t+b}dt-\gamma\left(  0,a\right)  \int_{1-b}^{x}\frac{\log^{m}\left(
t+b\right)  }{t+b}dt.
\end{align*}
Employing the integration by parts we see that%
\begin{align*}
\int_{1-b}^{x}\frac{\log\left(  t+a\right)  \log^{m}\left(  t+b\right)  }%
{t+b}dt  &  =\frac{\log^{m+2}\left(  x+b\right)  }{m+2}-m!\sum\limits_{k=1}%
^{\infty}\frac{\left(  b-a\right)  ^{k}}{k^{m+2}}\\
&  +\sum\limits_{k=1}^{\infty}\frac{\left(  b-a\right)  ^{k}}{\left(
x+b\right)  ^{k}}\sum\limits_{j=0}^{m}\frac{m!\log^{m-j}\left(  x+b\right)
}{k^{j+2}\left(  m-j\right)  !}.
\end{align*}
Combining these results we arrive at%
\begin{align*}
&  \gamma_{H,1}\left(  m,a,b\right) \\
&  \ =\sum\limits_{n\leq x}\frac{H_{n}\left(  a\right)  \log^{m}\left(
n+b\right)  }{n+b}-\frac{\log^{m+2}\left(  x+b\right)  }{m+2}-\gamma\left(
0,a\right)  \frac{\log^{m+1}\left(  x+b\right)  }{m+1}\\
&  \ -\sum\limits_{k=1}^{\infty}\sum\limits_{j=0}^{m}\frac{\left(  b-a\right)
^{k}}{k^{j+2}\left(  x+b\right)  ^{k}}\frac{m!\log^{m-j}\left(  x+b\right)
}{\left(  m-j\right)  !}+o\left(  1\right)  ,
\end{align*}
which gives the desired result as $x\rightarrow\infty$.
\end{proof}

\section{The harmonic digamma function and a summation
formula\label{secDigamma}}

The classical identity
\[
\sum\limits_{m=2}^{\infty}\left(  -1\right)  ^{m}\zeta\left(  m,a\right)
t^{m-1}=\psi\left(  a+t\right)  -\psi\left(  a\right)  ,
\]
shows that the Hurwitz zeta values at positive integers occur as the Taylor
coefficients of the digamma function. Thus, the digamma function encodes the
Hurwitz zeta values at positive integers into a single analytic function. It
is therefore natural to ask whether an analogous function exists for the
parametric harmonic zeta function. This motivates the introduction of the
harmonic digamma function.

The following proposition provides the key analytic ingredient for the
construction of such a function.

\begin{proposition}
Let $0<a<2$ and $b\in\mathbb{R}\backslash\left\{  0,-1,-2,-3,\ldots\right\}
$. Then, for $N\in\mathbb{N}$ we have
\begin{align*}
\sum_{n=0}^{N}\dfrac{H_{n}\left(  a\right)  }{n+b}  &  =\frac{\log^{2}\left(
N+b\right)  }{2}-\psi\left(  a\right)  \log\left(  N+b\right)  -\psi
_{H}\left(  a,b\right)  +\mathrm{Li}_{2}\left(  \frac{b-a}{b+N}\right) \\
&  +\frac{\log\frac{N+a}{N+b}}{2\left(  b-a\right)  }-\overline{B}_{1}\left(
0\right)  \dfrac{H_{N}\left(  a\right)  }{N+b}+\frac{\overline{B}_{2}\left(
0\right)  }{2}\left(  \frac{\psi^{\prime}\left(  N+1+a\right)  }{N+b}%
-\frac{H_{N}\left(  a\right)  }{\left(  N+b\right)  ^{2}}\right) \\
&  +\frac{1}{2}%
{\displaystyle\int\limits_{N}^{\infty}}
f^{(2)}(u)\overline{B}_{2}\left(  u\right)  du+%
{\displaystyle\int\limits_{0}^{\infty}}
\frac{2\left(  a-b\right)  \left(  \arctan\left(  \frac{N+a}{t}\right)
-\frac{\pi}{2}\right)  }{\left(  e^{2\pi t}-1\right)  \left(  \left(
a-b\right)  ^{2}+t^{2}\right)  }dt\\
&  +%
{\displaystyle\int\limits_{0}^{\infty}}
\frac{\log\left(  \frac{t^{2}}{\left(  N+b\right)  ^{2}}+1\right)  }{\left(
e^{2\pi t}-1\right)  \left(  \left(  a-b\right)  ^{2}+t^{2}\right)  }%
tdt+\frac{\delta_{a}}{ab}%
\end{align*}
where $\delta_{a}=1$ if $0<a\leq1$ and $\delta_{a}=0$ if $1<a<2$, and
\begin{equation}
\psi_{H}\left(  a,b\right)  :=-\lim_{N\rightarrow\infty}\left(  \sum_{n=0}%
^{N}\dfrac{H_{n}\left(  a\right)  }{n+b}-\frac{\log^{2}\left(  N+b\right)
}{2}+\psi\left(  a\right)  \log\left(  N+b\right)  \right)  \label{hdigamma}%
\end{equation}
(that we propose to call harmonic digamma function).
\end{proposition}

\begin{proof}
We make use of the Euler-Maclaurin summation formula in the following form
\cite[Theorem 9.2.2]{C}: Let $\alpha$ and $\beta$ be real numbers such that
$\alpha\leq\beta$ and assume that $f\in C^{\left(  l\right)  }\left[
\alpha,\beta\right]  $ for some $l\geq1.$ Then%
\begin{align}
\sum_{\alpha<m\leq\beta}f(m)  &  =%
{\displaystyle\int\limits_{\alpha}^{\beta}}
f(u)du+\sum\limits_{j=1}^{l}\frac{(-1)^{j}}{j!}\left(  \overline{B}_{j}\left(
\beta\right)  f^{(j-1)}(\beta)-\overline{B}_{j}\left(  \alpha\right)
f^{(j-1)}(\alpha)\right) \nonumber\\
&  +\frac{(-1)^{l-1}}{l!}%
{\displaystyle\int\limits_{\alpha}^{\beta}}
f^{(l)}(u)\overline{B}_{l}\left(  u\right)  du. \label{E-M}%
\end{align}

Let $f\left(  x\right)  =\dfrac{\psi\left(  x+1+a\right)  -\psi\left(
a\right)  }{x+b}$ in (\ref{E-M}) with $\alpha=1-a$ and $\beta=N$. Then,%
\begin{align*}
\sum_{n=0}^{N}\dfrac{H_{n}\left(  a\right)  }{n+b}  &  =%
{\displaystyle\int\limits_{1-a}^{N}}
\dfrac{\psi\left(  x+1+a\right)  -\psi\left(  a\right)  }{x+b}dx+\overline
{B}_{1}\left(  -a\right)  \dfrac{\psi\left(  2\right)  -\psi\left(  a\right)
}{1+b-a}\\
&  -\overline{B}_{1}\left(  0\right)  \dfrac{H_{N}\left(  a\right)  }%
{N+b}-\overline{B}_{2}\left(  -a\right)  \left(  \frac{\psi^{\prime}\left(
2\right)  }{1+b-a}-\frac{\psi\left(  2\right)  -\psi\left(  a\right)
}{\left(  1+b-a\right)  ^{2}}\right) \\
&  +\frac{\overline{B}_{2}\left(  0\right)  }{2}\left(  \frac{\psi^{\prime
}\left(  N+1+a\right)  }{N+b}-\frac{H_{N}\left(  a\right)  }{\left(
N+b\right)  ^{2}}\right) \\
&  -\frac{1}{2}%
{\displaystyle\int\limits_{1-a}^{\infty}}
f^{(2)}(u)\overline{B}_{2}\left(  u\right)  du+\frac{1}{2}%
{\displaystyle\int\limits_{N}^{\infty}}
f^{(2)}(u)\overline{B}_{2}\left(  u\right)  du+\frac{\delta_{a}}{ab}.
\end{align*}
We now use
\[
\psi\left(  x\right)  =\log\left(  x\right)  -\frac{1}{2x}-2%
{\displaystyle\int\limits_{0}^{\infty}}
\dfrac{t}{\left(  t^{2}+x^{2}\right)  \left(  e^{2\pi t}-1\right)  }dt,\text{
}\operatorname{Re}\left(  x\right)  >0,
\]
(cf. \cite[Chp. 1.3]{SC}) and deduce that%
\begin{align*}%
{\displaystyle\int\limits_{1-a}^{N}}
\dfrac{\psi\left(  x+1+a\right)  }{x+b}dx  &  =%
{\displaystyle\int\limits_{1-a}^{N}}
\dfrac{\log\left(  x+a\right)  }{x+b}dx+\frac{1}{2}%
{\displaystyle\int\limits_{1-a}^{N}}
\dfrac{dx}{\left(  x+b\right)  \left(  x+a\right)  }\\
&  -2%
{\displaystyle\int\limits_{0}^{\infty}}
\frac{t}{e^{2\pi t}-1}%
{\displaystyle\int\limits_{1-a}^{N}}
\dfrac{1}{\left(  x+b\right)  \left(  t^{2}+\left(  x+a\right)  ^{2}\right)
}dxdt.
\end{align*}
Here the integrals on the RHS are
\begin{align*}
\int\limits_{1-a}^{N}\frac{\log\left(  t+a\right)  dt}{\left(  t+b\right)  }
&  =\frac{\log^{2}\left(  N+b\right)  }{2}-\frac{\log^{2}\left(  1+b-a\right)
}{2}\\
&  +\mathrm{Li}_{2}\left(  \frac{b-a}{b+N}\right)  -\mathrm{Li}_{2}\left(
\frac{b-a}{1+b-a}\right)  ,
\end{align*}
and%
\begin{align*}
{\int\limits_{1-a}^{N}}\dfrac{1}{\left(  x+b\right)  \left(  t^{2}+\left(
x+a\right)  ^{2}\right)  }dx  &  =\frac{\log\left(  N+b\right)  -\log\left(
1+b-a\right)  }{\left(  a-b\right)  ^{2}+t^{2}}\\
&  -\frac{\arctan\frac{N+a}{t}-\arctan\frac{1}{t}}{\left(  a-b\right)
^{2}+t^{2}}\frac{a-b}{t}-\frac{1}{2}\frac{\log\frac{t^{2}+\left(  N+b\right)
^{2}}{t^{2}+1}}{\left(  a-b\right)  ^{2}+t^{2}}.
\end{align*}
Thus, we arrive at
\begin{align}
\sum_{n=0}^{N}\dfrac{H_{n}\left(  a\right)  }{n+b}  &  =\frac{\log^{2}\left(
N+b\right)  }{2}-\psi\left(  a\right)  \log\left(  N+b\right)  +\psi\left(
a\right)  \log\left(  1+b-a\right) \nonumber\\
&  -\frac{\log^{2}\left(  1+b-a\right)  }{2}+\frac{\overline{B}_{2}\left(
0\right)  }{2}\left(  \frac{\psi^{\prime}\left(  N+1+a\right)  }{N+b}%
-\frac{H_{N}\left(  a\right)  }{\left(  N+b\right)  ^{2}}\right) \nonumber\\
&  +\mathrm{Li}_{2}\left(  \frac{b-a}{b+N}\right)  -\mathrm{Li}_{2}\left(
\frac{b-a}{1-a+b}\right)  +%
{\displaystyle\int\limits_{0}^{\infty}}
\frac{2t\log\left(  1+b-a\right)  }{\left(  e^{2\pi t}-1\right)  \left(
\left(  a-b\right)  ^{2}+t^{2}\right)  }dt\nonumber\\
&  +\frac{\log\frac{N+a}{N+b}+\log\left(  1+b-a\right)  }{2\left(  b-a\right)
}+%
{\displaystyle\int\limits_{0}^{\infty}}
\frac{\log\left(  \frac{t^{2}}{\left(  N+b\right)  ^{2}}+1\right)
-\log\left(  t^{2}+1\right)  }{\left(  e^{2\pi t}-1\right)  \left(  \left(
a-b\right)  ^{2}+t^{2}\right)  }tdt\nonumber\\
&  -\overline{B}_{1}\left(  0\right)  \dfrac{H_{N}\left(  a\right)  }%
{N+b}+2\left(  a-b\right)
{\displaystyle\int\limits_{0}^{\infty}}
\frac{\arctan\left(  \frac{N+a}{t}\right)  -\arctan\left(  \frac{1}{t}\right)
}{\left(  e^{2\pi t}-1\right)  \left(  \left(  a-b\right)  ^{2}+t^{2}\right)
}dt\nonumber\\
&  +\overline{B}_{1}\left(  -a\right)  \dfrac{\psi\left(  2\right)
-\psi\left(  a\right)  }{1+b-a}-\overline{B}_{2}\left(  -a\right)  \left(
\frac{\psi^{\prime}\left(  2\right)  }{\left(  1+b-a\right)  }-\frac{\left(
\psi\left(  2\right)  -\psi\left(  a\right)  \right)  }{\left(  1+b-a\right)
^{2}}\right) \nonumber\\
&  -\frac{1}{2}%
{\displaystyle\int\limits_{1-a}^{\infty}}
f^{(2)}(u)\overline{B}_{2}\left(  u\right)  du+\frac{1}{2}%
{\displaystyle\int\limits_{N}^{\infty}}
f^{(2)}(u)\overline{B}_{2}\left(  u\right)  du+\frac{\delta_{a}}{ab}.
\label{30}%
\end{align}
On the other hand, the identity
\begin{align*}
\frac{1}{2}{\int\limits_{1-a}^{N}}f^{(2)}(u)\overline{B}_{2}\left(  u\right)
du  &  =-\sum_{n=0}^{N}\dfrac{H_{n}\left(  a\right)  }{n+b}+\frac{\log
^{2}\left(  N+b\right)  }{2}-\psi\left(  a\right)  \log\left(  N+b\right)
+\frac{\delta_{a}}{ab}\\
&  +\psi\left(  a\right)  \log\left(  1+b-a\right)  -\frac{\log^{2}\left(
1+b-a\right)  }{2}-\mathrm{Li}_{2}\left(  \frac{b-a}{1-a+b}\right) \\
&  +\frac{\log\frac{N+a}{N+b}+\log\left(  1+b-a\right)  }{2\left(  b-a\right)
}+2%
{\displaystyle\int\limits_{0}^{\infty}}
\frac{\log\left(  1+b-a\right)  }{\left(  e^{2\pi t}-1\right)  \left(  \left(
a-b\right)  ^{2}+t^{2}\right)  }tdt\\
&  +\mathrm{Li}_{2}\left(  \frac{b-a}{b+N}\right)  +2\left(  a-b\right)
{\displaystyle\int\limits_{0}^{\infty}}
\frac{\arctan\left(  \frac{N+a}{t}\right)  -\arctan\left(  \frac{1}{t}\right)
}{\left(  e^{2\pi t}-1\right)  \left(  \left(  a-b\right)  ^{2}+t^{2}\right)
}dt\\
&  +\overline{B}_{1}\left(  -a\right)  \dfrac{\psi\left(  2\right)
-\psi\left(  a\right)  }{1+b-a}+%
{\displaystyle\int\limits_{0}^{\infty}}
\frac{\log\left(  \frac{t^{2}}{\left(  N+b\right)  ^{2}}+1\right)
-\log\left(  t^{2}+1\right)  }{\left(  e^{2\pi t}-1\right)  \left(  \left(
a-b\right)  ^{2}+t^{2}\right)  }tdt\\
&  -\overline{B}_{1}\left(  0\right)  \dfrac{H_{N}\left(  a\right)  }%
{N+b}+\frac{\overline{B}_{2}\left(  0\right)  }{2}\left(  \frac{\psi^{\prime
}\left(  N+1+a\right)  }{N+b}-\frac{H_{N}\left(  a\right)  }{\left(
N+b\right)  ^{2}}\right) \\
&  -\overline{B}_{2}\left(  -a\right)  \left(  \frac{\psi^{\prime}\left(
2\right)  }{\left(  1+b-a\right)  }-\frac{\left(  \psi\left(  2\right)
-\psi\left(  a\right)  \right)  }{\left(  1+b-a\right)  ^{2}}\right)
\end{align*}
yields%
\begin{align}
&  \frac{1}{2}{\int\limits_{1-a}^{\infty}}f^{(2)}(u)\overline{B}_{2}\left(
u\right)  du\nonumber\\
&  =-\lim_{N\rightarrow\infty}\left(  \sum_{n=0}^{N}\dfrac{H_{n}\left(
a\right)  }{n+b}-\frac{\log^{2}\left(  N+b\right)  }{2}+\psi\left(  a\right)
\log\left(  N+b\right)  \right)  +\frac{\delta_{a}}{ab}\nonumber\\
&  +\psi\left(  a\right)  \log\left(  1+b-a\right)  -\frac{\log^{2}\left(
1+b-a\right)  }{2}-\mathrm{Li}_{2}\left(  \frac{b-a}{1-a+b}\right) \nonumber\\
&  +\frac{\log\left(  1+b-a\right)  }{2\left(  b-a\right)  }+2%
{\displaystyle\int\limits_{0}^{\infty}}
\frac{\log\left(  1+b-a\right)  }{\left(  e^{2\pi t}-1\right)  \left(  \left(
a-b\right)  ^{2}+t^{2}\right)  }tdt\nonumber\\
&  +2\left(  a-b\right)
{\displaystyle\int\limits_{0}^{\infty}}
\frac{\frac{\pi}{2}-\arctan\left(  \frac{1}{t}\right)  }{\left(  e^{2\pi
t}-1\right)  \left(  \left(  a-b\right)  ^{2}+t^{2}\right)  }dt-%
{\displaystyle\int\limits_{0}^{\infty}}
\frac{t\log\left(  t^{2}+1\right)  }{\left(  e^{2\pi t}-1\right)  \left(
\left(  a-b\right)  ^{2}+t^{2}\right)  }dt\nonumber\\
&  +\overline{B}_{1}\left(  -a\right)  \dfrac{\psi\left(  2\right)
-\psi\left(  a\right)  }{1+b-a}-\overline{B}_{2}\left(  -a\right)  \left(
\frac{\psi^{\prime}\left(  2\right)  }{\left(  1+b-a\right)  }-\frac{\left(
\psi\left(  2\right)  -\psi\left(  a\right)  \right)  }{\left(  1+b-a\right)
^{2}}\right)  . \label{31}%
\end{align}
Hence, from (\ref{30}) and (\ref{31}), it follows that%
\begin{align*}
\sum_{n=0}^{N}\dfrac{H_{n}\left(  a\right)  }{n+b}  &  =-\psi_{H}\left(
a,b\right)  +\frac{\log^{2}\left(  N+b\right)  }{2}-\psi\left(  a\right)
\log\left(  N+b\right)  +\frac{\delta_{a}}{ab}\\
&  +\mathrm{Li}_{2}\left(  \frac{b-a}{b+N}\right)  +\frac{\log\frac{N+a}{N+b}%
}{b-a}+\frac{\overline{B}_{2}\left(  0\right)  }{2}\left(  \frac{\psi^{\prime
}\left(  N+1+a\right)  }{N+b}-\frac{H_{N}\left(  a\right)  }{\left(
N+b\right)  ^{2}}\right) \\
&  +\frac{1}{2}%
{\displaystyle\int\limits_{N}^{\infty}}
f^{(2)}(u)\overline{B}_{2}\left(  u\right)  du+2\left(  a-b\right)
{\displaystyle\int\limits_{0}^{\infty}}
\frac{\arctan\left(  \frac{N+a}{t}\right)  -\frac{\pi}{2}}{\left(  e^{2\pi
t}-1\right)  \left(  \left(  a-b\right)  ^{2}+t^{2}\right)  }dt\\
&  -\overline{B}_{1}\left(  0\right)  \dfrac{H_{N}\left(  a\right)  }{N+b}+%
{\displaystyle\int\limits_{0}^{\infty}}
\frac{\log\left(  \frac{t^{2}}{\left(  N+b\right)  ^{2}}+1\right)  }{\left(
e^{2\pi t}-1\right)  \left(  \left(  a-b\right)  ^{2}+t^{2}\right)  }tdt,
\end{align*}
which is the desired assertion.
\end{proof}

Remark that the function $\psi_{H}\left(  a,a\right)  $ corresponds to\emph{
the harmonic Stieltjes constant} $-\gamma_{H}\left(  0,a\right)  $, which
appears in the Laurent expansion of $\zeta_{H}\left(  s,a\right)  $ in a
neighborhood of $s=1$ (see \cite[Theorem 1]{KDCC}), and $\psi_{H}\left(
1,1\right)  $ corresponds to \emph{the harmonic Stieltjes constant}
$-\gamma_{H}\left(  0\right)  $ occurring in the Laurent expansion of
$\zeta_{H}\left(  s\right)  $ in a neighborhood of $s=1$ (see \cite[p. 8]{CC}).

The harmonic digamma function $\psi_{H}\left(  a,b\right)  $ defined by
(\ref{hdigamma}) shares several properties with the classical digamma function.

\begin{proposition}
We have%
\begin{equation}
\psi_{H}\left(  a,b\right)  =-\lim_{N\rightarrow\infty}\left(  \sum_{n=0}%
^{N}\dfrac{H_{n}\left(  a\right)  }{n+b}-\frac{\log^{2}N}{2}+\psi\left(
a\right)  \log N\right)  \label{hd2}%
\end{equation}
and%
\begin{align}
\psi_{H}\left(  a,b\right)   &  =\psi_{H}\left(  a,1\right)  +\left(
b-1\right)  \sum_{n=0}^{\infty}\frac{H_{n}\left(  a\right)  }{\left(
n+1\right)  \left(  n+b\right)  }\nonumber\\
&  =\psi_{H}\left(  a,1\right)  +\sum_{n=0}^{\infty}\left(  \frac{H_{n}\left(
a\right)  }{n+1}-\frac{H_{n}\left(  a\right)  }{n+b}\right)  . \label{hd3}%
\end{align}
Moreover, for $m\in\mathbb{N}$,
\begin{align}
\psi_{H}^{\left(  m\right)  }\left(  a,b\right)   &  =\left(  -1\right)
^{m+1}m!\zeta_{H}\left(  m+1,a,b\right)  ,\label{hdt}\\
\psi_{H}^{\left(  m\right)  }\left(  a,q\right)   &  =\psi_{H}^{\left(
m\right)  }\left(  a,1\right)  +\lambda_{a}\left(  -1\right)  ^{m}%
m!\zeta\left(  m+2,a\right) \nonumber\\
&  +\left(  -1\right)  ^{m}m!\sum_{\substack{b=1\\b\not =a}}^{q-1}\left(
\frac{\gamma_{0}\left(  b\right)  -\gamma_{0}\left(  a\right)  }{\left(
b-a\right)  ^{m+1}}-\sum\limits_{j=2}^{m+1}\frac{\zeta\left(  j,b\right)
}{\left(  b-a\right)  ^{m+2-j}}\right)  , \label{hd4}%
\end{align}
and
\begin{equation}
\psi_{H}\left(  a,b\right)  =\psi_{H}\left(  a,1\right)  +\frac{1}%
{\Gamma\left(  s\right)  }\int_{0}^{\infty}\frac{e^{-x}-e^{-xb}}{1-e^{-x}}%
\Phi\left(  e^{-x},1;a\right)  dx, \label{hdi}%
\end{equation}
where $\psi_{H}^{\left(  m\right)  }\left(  a,x\right)  =\left.  \left(
\frac{\partial}{\partial t}\right)  ^{m}\psi_{H}\left(  a,t\right)
\right\vert _{t=x}$, $\Phi\left(  x,s;a\right)  =\sum_{k=0}^{\infty}%
x^{k}\left(  k+a\right)  ^{-s}$, and $\lambda_{a}=\left\{
\begin{array}
[c]{ll}%
1, & a\in\left\{  1,2,\ldots,q-1\right\}  \text{,}\\
0, & \text{otherwise.}%
\end{array}
\right.  $
\end{proposition}

\begin{proof}
The assertions (\ref{hd2}) and (\ref{hd3}) follow from (\ref{hdigamma}) and
(\ref{hd2}), respectively. For the assertion (\ref{hdt}), we differentiate
both sides of (\ref{hd3}) with respect to $b$ and see that
\[
\frac{\partial}{\partial b}\psi_{H}\left(  a,b\right)  =\left(  -1\right)
^{2}\sum_{n=0}^{\infty}\frac{H_{n}\left(  a\right)  }{\left(  n+b\right)
^{2}}.
\]
This implies (\ref{hdt}). The interchange of the order of summation and
derivative can be justified by the absolute convergence of the series.

For the assertion (\ref{hd4}), we write (\ref{hdt}) as
\[
\psi_{H}^{\left(  m\right)  }\left(  a,b+1\right)  =\psi_{H}^{\left(
m\right)  }\left(  a,b\right)  +\left(  -1\right)  ^{m}m!\sum_{n=0}^{\infty
}\frac{1}{\left(  n+b\right)  ^{m+1}\left(  n+a\right)  }%
\]
and sum over $b$ from $1$ to $q-1$. Then, (\ref{hd4}) follows from partial
fraction decomposition and%
\[
H_{N}\left(  a\right)  =\log\left(  N+a\right)  +\gamma_{0}\left(  a\right)
+O\left(  \frac{1}{N+a}\right)  ,\text{ as }N\rightarrow\infty.
\]

The assertion (\ref{hdi}) follows from (\ref{hdt}) and%
\[
\zeta_{H}\left(  2,a,b\right)  =\frac{1}{\Gamma\left(  s\right)  }\int%
_{0}^{\infty}\frac{xe^{-xb}}{1-e^{-x}}\Phi\left(  e^{-x},1;a\right)  dx.
\]

\end{proof}

As a consequence of (\ref{hd3}), we can state the following result regarding
the analyticity of $\psi_{H}\left(  a,z\right)  $.

\begin{corollary}
The function $\psi_{H}\left(  a,z\right)  $ is analytic for all $z$ except for
simple poles at $z=-v,$ $v\in\mathbb{N}\cup\left\{  0\right\}  $, with their
respective residues $\left(  -v-1\right)  H_{v}\left(  a\right)  .$
\end{corollary}

From (\ref{hd3}), we can also derive the following difference formula:

\begin{corollary}
Let $m\in\mathbb{N}$. Then we have
\begin{align*}
&  \psi_{H}\left(  a,b+m\right)  -\psi_{H}\left(  a,b\right) \\
&  =\left\{
\begin{array}
[c]{ll}%
{\displaystyle\sum\limits_{k=0}^{m-1}}
\dfrac{H_{k}\left(  a\right)  }{k+b}+%
{\displaystyle\sum\limits_{k=0}^{m-1}}
\dfrac{\psi\left(  b+m\right)  -\psi\left(  a+m-k\right)  }{b-a+k}, &
a\not =b,\medskip\\%
{\displaystyle\sum\limits_{k=0}^{m-1}}
\dfrac{H_{k}\left(  a\right)  }{k+a}+%
{\displaystyle\sum\limits_{k=1}^{m-1}}
\dfrac{\psi\left(  a+m\right)  -\psi\left(  a+m-k\right)  }{k}+\psi^{\prime
}\left(  a+m\right)  , & a=b.
\end{array}
\right.
\end{align*}
In particular, for $m=1$,
\begin{equation}
\psi_{H}\left(  a,b+1\right)  -\psi_{H}\left(  a,b\right)  =\left\{
\begin{array}
[c]{ll}%
\dfrac{\psi\left(  b\right)  -\psi\left(  a\right)  }{b-a}, & a\not =%
b,\medskip\\
\zeta\left(  2,a\right)  , & a=b.
\end{array}
\right.  \label{hd5}%
\end{equation}

\end{corollary}

In addition to (\ref{hd5}), for any $0<b\in\mathbb{R}$, it follows from
(\ref{hd3}) that
\[
\psi_{H}\left(  a,b+1\right)  -\psi_{H}\left(  a,b\right)  =\sum_{k=0}%
^{\infty}\left(  -1\right)  ^{k}\zeta_{H}\left(  k+2,a,b\right)  ,
\]
provided that the series converges. Thus, we have%
\[
\sum_{k=0}^{\infty}\left(  -1\right)  ^{k}\zeta_{H}\left(  k+2,a,b\right)
=\left\{
\begin{array}
[c]{ll}%
\dfrac{\psi\left(  b\right)  -\psi\left(  a\right)  }{b-a}, & a\not =b,\\
\zeta\left(  2,a\right)  , & a=b.
\end{array}
\right.
\]
In the case $a\not =b$, differentiating both sides with respect to $b$ and $a$
gives
\[
\sum_{k=0}^{\infty}\left(  -1\right)  ^{k+1}\left(  k+2\right)  \zeta
_{H}\left(  k+3,a,b\right)  =\dfrac{\zeta\left(  2,b\right)  }{b-a}%
-\dfrac{\psi\left(  b\right)  -\psi\left(  a\right)  }{\left(  b-a\right)
^{2}},
\]
and
\[
\sum_{k=0}^{\infty}\left(  -1\right)  ^{k+1}\zeta_{H^{\left(  2\right)  }%
}\left(  k+2,a,b\right)  =\dfrac{\zeta\left(  2,a\right)  }{a-b}+\dfrac
{\psi\left(  b\right)  -\psi\left(  a\right)  }{\left(  b-a\right)  ^{2}},
\]
respectively.

The harmonic digamma function has the following Taylor expansion:

\begin{corollary}
We have
\begin{equation}
\sum\limits_{m=2}^{\infty}\left(  -1\right)  ^{m}\zeta_{H}\left(
m,a,b\right)  t^{m-1}=\psi_{H}\left(  a,b+t\right)  -\psi_{H}\left(
a,b\right)  . \label{dgtaylor}%
\end{equation}

\end{corollary}

The techniques developed above also yield the following parametric harmonic
analogue of the summation formula
\begin{align}
\sum\limits_{m=2}^{\infty}\zeta\left(  m,a\right)  \frac{z^{m+p}}{m+p}  &
=\sum_{j=0}^{p}\binom{p}{j}\zeta^{\prime}\left(  -j,a-z\right)  z^{p-j}%
-\sum_{j=0}^{p-1}\frac{\zeta\left(  -j,a\right)  }{p-k}z^{p-j}\nonumber\\
&  +\left(  \psi\left(  a\right)  -H_{p}\right)  \frac{z^{p+1}}{p+1}%
-\zeta^{\prime}\left(  -p,a\right)  ,\text{ }\left\vert z\right\vert
<\left\vert a\right\vert ,\text{ }p\in\mathbb{N}\cup\left\{  0\right\}  .
\label{sf}%
\end{align}
This formula was first recorded by Kanemitsu et al. \cite{KKY} (see also
\cite{DMC,KKSY,MD,SC}). The case $a=1$ was composed by Ramanujan \cite[Entry
28(b)]{Be2} (see also \cite{Ada,KCe,SC}).

\begin{theorem}
For a nonnegative integer $p$, we have%
\begin{align}
\sum\limits_{m=2}^{\infty}\zeta_{H}\left(  m,a,b\right)  \frac{\left(
-t\right)  ^{m+p}}{m+p}  &  =\sum_{j=0}^{p-1}\psi_{H}^{\left(  -j-1\right)
}\left(  a,b+t\right)  \left\langle p\right\rangle _{j}\left(  -t\right)
^{p-j}+\frac{\psi_{H}\left(  a,b\right)  }{p+1}\left(  -t\right)
^{p+1}\nonumber\\
&  +p!\left(  \psi_{H}^{\left(  -p-1\right)  }\left(  a,b+t\right)  -\psi
_{H}^{\left(  -p-1\right)  }\left(  a,b\right)  \right)  , \label{hzsf}%
\end{align}
where%
\[
\psi_{H}^{\left(  -p-1\right)  }\left(  a,t\right)  =\int\psi_{H}^{\left(
-p\right)  }\left(  a,t\right)  dt\text{ with }\psi_{H}^{\left(  -1\right)
}\left(  a,t\right)  =\int\psi_{H}\left(  a,t\right)  dt.
\]

\end{theorem}

\begin{proof}
We first multiply (\ref{dgtaylor}) by $t^{p}$ and then integrate with respect
to $t$ from $0$ to $z$:%
\[
\sum\limits_{m=2}^{\infty}\left(  -1\right)  ^{m}\zeta_{H}\left(
m,a,b\right)  \frac{z^{m+p}}{m+p}=\int\limits_{0}^{z}\psi_{H}\left(
a,b+t\right)  t^{p}dt-\psi_{H}\left(  a,b\right)  \frac{z^{p+1}}{p+1}.
\]
Twice integrating by parts gives
\begin{align*}
\int\limits_{0}^{z}\psi_{H}\left(  a,b+t\right)  t^{p}dt  &  =z^{p}\psi
_{H}^{\left(  -1\right)  }\left(  a,b+z\right)  +\left(  -1\right)
^{1}pz^{p-1}\psi_{H}^{\left(  -2\right)  }\left(  a,b+z\right) \\
&  \quad+\left(  -1\right)  ^{2}p\left(  p-1\right)  \int\limits_{0}^{z}%
\psi_{H}^{\left(  -2\right)  }\left(  a,b+t\right)  t^{p-2}dt.
\end{align*}
Repeating this procedure $\left(  p-2\right)  $ times completes the proof
\end{proof}

Building upon these techniques, we now establish a further identity.

\begin{theorem}
For $p\in\mathbb{N}$, we have
\begin{align}
&  \sum_{j=0}^{\infty}\left(  -1\right)  ^{j-1}\frac{\left(  s\right)  _{j}%
}{j!}\zeta_{H}\left(  s+j,a,b\right)  \frac{x^{j+p}}{j+p}\nonumber\\
&  =\sum\limits_{v=1}^{p-1}\frac{\left\langle p-1\right\rangle _{v-1}%
}{\left\langle s-1\right\rangle _{v}}\zeta_{H}\left(  s-v,a,b+x\right)
x^{p-v}-\frac{\left\langle p-1\right\rangle _{p-1}}{\left\langle
s-1\right\rangle _{p-1}}\sigma\left(  s,p,a,b,x\right)  , \label{05}%
\end{align}
where
\[
\sigma\left(  s,p,a,b,x\right)  =\left\{
\begin{array}
[c]{ll}%
\dfrac{1}{p-s}\left\{  \zeta_{H}\left(  s-p,a,b+x\right)  -\zeta_{H}\left(
s-p,a,b\right)  \right\}  , & \operatorname{Re}\left(  s-p\right)
>1,\smallskip\\
\psi_{H}\left(  a,b+x\right)  -\psi_{H}\left(  a,b\right)  , & s-p=1.
\end{array}
\right.
\]

\end{theorem}

\begin{proof}
From (\ref{04}) we have%
\[
\sum_{j=0}^{\infty}\left(  -1\right)  ^{j}\frac{\left(  s\right)  _{j}}%
{j!}\zeta_{H}\left(  s+j,a,b\right)  \frac{x^{j+p}}{j+p}=\int\limits_{0}%
^{x}\zeta_{H}\left(  s,a,b+t\right)  t^{p-1}dt.
\]
Let
\[
I\left(  s,p-1\right)  =\int\limits_{0}^{x}\zeta_{H}\left(  s,a,b+t\right)
t^{p-1}dt.
\]
Integration by parts gives%
\[
I\left(  s,p-1\right)  =-\frac{x^{p-1}}{s-1}\zeta_{H}\left(  s-1,a,b+x\right)
+\frac{p-1}{s-1}I\left(  s-1,p-2\right)  .
\]
Employing this reduction formula $p-2$ additional times yields%
\begin{align*}
I\left(  s,p-1\right)   &  =-\sum\limits_{v=1}^{p-1}\frac{\left\langle
p-1\right\rangle _{v-1}}{\left\langle s-1\right\rangle _{v}}\zeta_{H}\left(
s-v,a,b+x\right)  x^{p-v}\\
&  +\frac{\left\langle p-1\right\rangle _{p-1}}{\left\langle s-1\right\rangle
_{p-1}}\int\limits_{0}^{x}\zeta_{H}\left(  s-p+1,a,b+t\right)  dt.
\end{align*}
Thus, the desired result follows from
\[
\int\limits_{0}^{x}\zeta_{H}\left(  s+1-p,a,b+t\right)  dt=\sigma\left(
s,p,a,b,x\right)  ,
\]
upon the use of (\ref{hdt}).
\end{proof}

We conclude the paper with the following remarks: In particular cases, the
summation formula (\ref{hzsf}) yields summation formulas for $\zeta_{H}\left(
s,a\right)  $ and $\zeta_{H}\left(  s\right)  $. Further, as a consequence of
(\ref{hzsf}) one can deduce that%
\begin{align*}
\sum\limits_{m=2}^{\infty}\left(  -1\right)  ^{m}\zeta_{H}\left(
m,a,b\right)  \frac{t^{m}}{m}  &  =\psi_{H}^{\left(  -1\right)  }\left(
a,b+t\right)  -\psi_{H}^{\left(  -1\right)  }\left(  a,b\right)  -t\psi
_{H}\left(  a,b\right)  ,\\
\sum\limits_{m=1}^{\infty}\zeta_{H}\left(  2m,a,b\right)  \frac{t^{2m}}{m}  &
=\psi_{H}^{\left(  -1\right)  }\left(  a,b+t\right)  -2\psi_{H}^{\left(
-1\right)  }\left(  a,b\right)  +\psi_{H}^{\left(  -1\right)  }\left(
a,b-t\right)  ,\\
\sum\limits_{m=1}^{\infty}\zeta_{H}\left(  2m+1,a,b\right)  \frac{t^{2m+1}%
}{2m+1}  &  =\frac{1}{2}\psi_{H}^{\left(  -1\right)  }\left(  a,b-t\right)
-\frac{1}{2}\psi_{H}^{\left(  -1\right)  }\left(  a,b+t\right)  +t\psi
_{H}\left(  a,b\right)  ,
\end{align*}
and%
\begin{align*}
\sum\limits_{m=1}^{\infty}\zeta_{H}\left(  2m,a,b\right)  t^{2m-1}  &
=\frac{1}{2}\psi_{H}\left(  a,b+t\right)  -\frac{1}{2}\psi_{H}\left(
a,b-t\right)  ,\\
\sum\limits_{m=1}^{\infty}\zeta_{H}\left(  2m+1,a,b\right)  t^{2m}  &
=-\frac{1}{2}\psi_{H}\left(  a,b-t\right)  -\frac{1}{2}\psi_{H}\left(
a,b+t\right)  +\psi_{H}\left(  a,b\right)  .
\end{align*}
In addition, for suitable special values of argument $t$, further series
identities can be obtained from (\ref{04}), (\ref{05}), and (\ref{hzsf}).

\section{Multiplication formulas}

In this section, we present multiplication formulas for the harmonic zeta
function, harmonic Stieltjes constants and harmonic digamma function.

\begin{lemma}
Let $N\in\mathbb{N}$. For all $s\in\mathbb{C}\backslash\left\{  k\in
\mathbb{Z}:k\leq1\right\}  $, we have
\begin{equation}
\sum\limits_{k=0}^{N-1}\sum\limits_{j=0}^{N-1}\zeta_{H}\left(  s,\frac{a+j}%
{N},\frac{b+j+k}{N}\right)  =N^{s+1}\zeta_{H}\left(  s,a,b\right)  .
\label{raabezetahsab}%
\end{equation}

\end{lemma}

\begin{proof}
The multiplication formula (\ref{raabezetahsab} follows by starting from the
double sum and using the definition of $\zeta_{H}\left(  s,a,b\right)  $.
\end{proof}

\begin{proposition}
Let $N\in\mathbb{N}$ and $m\in\mathbb{N}\cup\left\{  0\right\}  $. The
constants $\gamma_{H,1}\left(  m,a,b\right)  $ obey the following
multiplication formula:%
\begin{align}
&  \frac{\left(  -1\right)  ^{n}}{N^{2}}\sum\limits_{j=0}^{N-1}\sum
\limits_{k=0}^{N-1}\gamma_{H,1}\left(  n,\frac{a+j}{N},\frac{b+j+k}{N}\right)
\label{raabe}\\
&  \ =\frac{\left(  \log N\right)  ^{n+2}}{\left(  n+2\right)  \left(
n+1\right)  }+\frac{\gamma_{0}\left(  a\right)  \left(  \log N\right)  ^{n+1}%
}{n+1}+\sum\limits_{m=0}^{n}\left(  -1\right)  ^{m}\binom{n}{m}\left(  \log
N\right)  ^{n-m}\gamma_{H,1}\left(  m,a,b\right)  .\nonumber
\end{align}

\end{proposition}

\begin{proof}
For the LHS of (\ref{raabezetahsab}), we deduce from (\ref{laurentzetahsab})
that
\begin{align*}
\sum\limits_{j=0}^{N-1}\sum\limits_{k=0}^{N-1}\zeta_{H}\left(  s,\frac{a+j}%
{N},\frac{b+j+k}{N}\right)   &  =\frac{N^{2}}{\left(  s-1\right)  ^{2}}%
+\frac{N^{2}\gamma_{0}\left(  a\right)  +N^{2}\log N}{s-1}\\
&  +\sum\limits_{n=0}^{\infty}\sum\limits_{j=0}^{N-1}\sum\limits_{k=0}%
^{N-1}\frac{\left(  -1\right)  ^{n}}{n!}\gamma_{H,1}\left(  n,\frac{a+j}%
{N},\frac{b+j+k}{N}\right)  \left(  s-1\right)  ^{n},
\end{align*}
where we have used that
\[
\sum\limits_{j=0}^{N-1}\psi\left(  x+\frac{j}{N}\right)  =N\psi\left(
Nx\right)  -N\log N\text{, with }\psi\left(  a\right)  =-\gamma_{0}\left(
a\right)  .
\]
Again from (\ref{laurentzetahsab}), we deduce that the RHS of
(\ref{raabezetahsab}) can be written as
\begin{align*}
&  N^{s+1}\zeta_{H}\left(  s,a,b\right) \\
&  =\frac{N^{2}}{\left(  s-1\right)  ^{2}}+\frac{N^{2}}{s-1}\left(  \log
N+\gamma_{0}\left(  a\right)  \right)  +\sum\limits_{n=0}^{\infty}\left\{
\frac{N^{2}\left(  \log N\right)  ^{n+2}}{\left(  n+2\right)  !}\right. \\
&  \left.  +\frac{N^{2}\gamma_{0}\left(  a\right)  \left(  \log N\right)
^{n+1}}{\left(  n+1\right)  !}+N^{2}\sum\limits_{m=0}^{n}\frac{\left(  \log
N\right)  ^{n-m}}{\left(  n-m\right)  !}\frac{\left(  -1\right)  ^{m}%
\gamma_{H,1}\left(  m,a,b\right)  }{m!}\right\}  \left(  s-1\right)  ^{n}.
\end{align*}
These complete the proof.
\end{proof}

Considering that $\gamma_{H,1}\left(  0,a,b\right)  =-\psi_{H}\left(
a,b\right)  $, the special case of (\ref{raabe}) yields the Raabe relation for
the function $\psi_{H}\left(  a,b\right)  $.

\begin{corollary}
Let $N\in\mathbb{N}$. Then we have%
\begin{equation}
\frac{1}{N^{2}}\sum\limits_{j=0}^{N-1}\sum\limits_{k=0}^{N-1}\psi_{H}\left(
\frac{a+j}{N},\frac{b+j+k}{N}\right)  =\psi_{H}\left(  a,b\right)
+\psi\left(  a\right)  \log N-\frac{\left(  \log N\right)  ^{2}}{2}.
\label{raabe-digamma}%
\end{equation}

\end{corollary}

\end{document}